\documentclass[11pt]{article}
\usepackage{hyperref}
\usepackage{amsmath,amssymb,amsfonts,enumitem,xfrac,booktabs,mathabx,stmaryrd,wasysym}
\usepackage{amsthm,float}
\usepackage{wrapfig}
\usepackage{graphicx}
 \usepackage{multirow,tabularx}
 \usepackage{tikz}
 \usepackage[margin = 1in]{geometry}
\usepackage[mode=buildnew]{standalone}
\usepackage{subcaption}
\usetikzlibrary{intersections}
\usepackage{adjustbox}
\usepackage{pgfplotstable}
\usepackage{pgfplots}
\def\mylinewidth{0.30mm}
\def\mymarkersize{1.75}

\usepackage[numbers,sort]{natbib}
 \def\newblock{\ }%
 \bibpunct[, ]{[}{]}{,}{n}{}{,}%

\newcommand{\germansF}{Po\-si\-tiv\-ste\-llen\-s\"atze}
\newcommand{\germanF}{Positivstellensatz}
\newcommand{\germans}{Ps\"atze}
\newcommand{\german}{Psatz}
\newcommand{\schmu}{Schm\"{u}d\-gen}
\newcommand\R{{\mathbb R}}
\newcommand{\K}{\mathcal{K}}
\newcommand{\N}{\mathbb{N}}
\newcommand{\tr}{^\intercal}
\newcommand{\eps}{\varepsilon}
\newcommand{\inter}{\operatorname{int}}
\newcommand{\rec}[1]{{#1}^\infty}
\renewcommand{\P}{\mathcal{P}}
\renewcommand{\Box}{\operatorname{Box}}
\newcommand{\outM}{\multicolumn{1}{c}{$\blackdiamond$}}
\newcommand{\noVal}{\multicolumn{1}{c}{$\relbar$}}
\newcommand{\infble}{\multicolumn{1}{c}{infeas.}}
\newcommand{\outT}{\multicolumn{1}{c}{\clock}}
\newcommand{\outTime}{\clock}
\newcommand{\toutT}{\clock}
\newcommand{\toutM}{$\blackdiamond$}

\newcommand{\lb}{\operatorname{lb}}
\newcommand{\status}{sts.}

\newtheorem{remark}{Remark}
\newtheorem{example}{Example}
\newtheorem{theorem}{Theorem}
\newtheorem{lemma}{Lemma}
\newtheorem{proposition}{Proposition}
\newtheorem{corollary}{Corollary}
\newtheorem{definition}{Definition}

\title{
Reducing nonnegativity over general semialgebraic sets to nonnegativity over simple sets
}

\author{Olga Kuryatnikova\thanks{
              Econometric Institute, Erasmus University Rotterdam, Rotterdam, The Netherlands
               (\href{mailto:kuryatnikova@ese.eur.nl}{kuryatnikova@ese.eur.nl}, \url{https://www.kuryatnikova.com/}).}
\and Juan C. Vera\thanks{Tilburg School of Economics and Management, Tilburg University, Tilburg, The Netherlands
(\href{mailto:j.c.veralizcano@tilburguniversity.edu}{j.c.veralizcano@tilburguniversity.edu}, \url{https://www.tilburguniversity.edu/staff/j-c-veralizcano}.}
\and Luis F. Zuluaga\thanks{Industrial and Systems Engineering Department, Lehigh University, Bethlehem, PA, USA
(\href{mailto:luis.zuluaga@lehigh.edu}{luis.zuluaga@lehigh.edu}, \url{https://coral.ise.lehigh.edu/luiszuluaga/}).}
}
\begin{document}
\maketitle

% REQUIRED
\begin{abstract} A nonnegativity certificate (NNC) is a way to write a polynomial so that its nonnegativity on a semialgebraic set becomes evident. \germansF{} (\germans) guarantee the existence of NNCs. Both, NNCs and \germans{} underlie powerful algorithmic techniques for optimization.
This paper proposes a universal approach to derive new \germans{} for general semialgebraic sets from ones developed for simpler  sets, such as a box, a simplex, or the nonnegative orthant. We provide several results illustrating the approach.
 First, by considering Handelman's \germanF{} (\german{}) over a box, we construct {\em non-SOS Schm\"{u}dgen-type} \germans{} over any compact semialgebraic set. That is, a family of \germans{} that follow the structure of the fundamental Schm\"{u}dgen's \german{}, but where instead of SOS polynomials,  any class of polynomials containing the nonnegative constants can be used, such as SONC, DSOS/SDSOS, hyperbolic or sums of AM/GM polynomials.
 Secondly, by considering the simplex as the simple set, we derive a sparse \german{} over general compact sets, which does not rely on any structural assumptions of the set.
 Finally, by considering P\'olya's \german{} over the nonnegative orthant, we
  derive a new non-SOS \german{} over unbounded sets which satisfy some generic conditions.
 All these results contribute to the literature regarding the use of non-SOS polynomials and sparse NNCs to derive \germans{} over compact and unbounded sets.
 Throughout the article, we illustrate our results with relevant examples and numerical experiments.
\end{abstract}


\section{Introduction} \label{sec:intro}

A {\em nonnegativity certificate} (NNC)  is a way to write a given function such that its nonnegativity on a given set becomes evident. {\em \germansF{}} (\germans) assert the existence of certain classes of NNCs. Both NNCs and \germans{} are fundamental tools in optimization, and they underlie powerful algorithmic techniques for various types of optimization problems~\citep[see, e.g.,][]{lasserre2015introduction,anjo12}. NNCs and optimization are related through the identity
\begin{equation}\label{eq:NNC-OPT}
\min_x\{p(x):x\in S\}  = \max_{\lambda} \{\lambda: p(x) - \lambda \text{ nonnegative on } S\};
\end{equation}
that shows that optimization over the set $S \subseteq \R^n$  is {\em equivalent} to certifying nonnegativity on $S$.

We are interested in the case when $p$ is a polynomial and when $S$ is a semialgebraic set; that is a set defined by polynomial
inequalities (and equalities).
\germans{} for semialgebraic sets have been studied for more than a century in algebraic geometry~\citep[see, e.g.,][]{reznick1995uniform}.
Together with the dual theory of moments~\citep[see, e.g.,][]{curto2000truncated,schmudgen1991k}, they are the foundation of modern {\em polynomial optimization} (PO)~\citep[see, e.g.,][]{lasserre2015introduction, anjo12}, which studies optimization problems where both the objective and constraints are defined using polynomials.
Broadly speaking, since NNCs provide sufficient conditions for nonnegativity, {\em degree-restricted} versions of the NNC~\citep[see, e.g.,][]{lasserre2009moments} can be used to obtain a hierarchy of inner approximations of the feasible set in the right-hand side of~\eqref{eq:NNC-OPT}; which in turn provides a hierarchy of lower bounds for~\eqref{eq:NNC-OPT}.
 The ability to solve the resulting approximation problems depends on various characteristics of the NNC, including its number of terms and sparsity/symmetry structure~\citep[see, e.g.,][]{magron2022sparse}, as well as  the {\em base class} (of polynomials) used to construct it \citep[see, e.g.,][]{kuang2017alternative, roebers2021sparse}. By base class, we mean sets of nonnegative polynomials such as sum of squares (SOS),  {\em sums of nonnegative circuit} (SONC)~\citep[see, e.g.,][]{dressler2017positivstellensatz}, and {\em sums of scaled diagonal dominant SOS} (SDSOS)~\citep[see, e.g.,][]{fidalgo2011positive, ahmadi2019dsos} used to construct NNCs.

 Further, a \germanF{} (\german{}) implies the  convergence, to the PO problem's optimal value, of the sequence of lower bounds obtained from solving the corresponding degree-restricted  hierarchy~\citep[see, e.g.,][]{lasserre2009moments}.
Moreover, an explicit bound on the degree of the polynomials appearing in the \german{} implies a convergence rate for the sequence of lower bounds from the hierarchy~\citep[see, e.g.,][]{KlLaur10-degreeBounds}.

\subsection{Main contribution} \label{subsec:intro_main}
\label{sec:approach}

As the title of the article asserts, our main contribution is to show how to take advantage of \germans{} on simple sets (e.g., a box or a simplex), to construct novel \germans{} on general semialgebraic sets; thus, providing an alternative to the existing mathematical strategies to derive \germans{}. As an illustration, we obtain \germans{} that are novel in terms of their sparsity structure, and the base classes (e.g., SOS, SONC, and SDSOS polynomials, among others).

Our approach is inspired by the following observation. Given $g \in \R[x]^m$, the complexity of the underlying semialgebraic set~$S_g := \{x \in \R^n: g_1(x) \ge 0, \dots, g_m(x) \ge 0\}$ plays an important role in the known theory about NNCs over~$S_g$.
For instance, when $S_g$ is compact, \schmu{}'s \german{}~\citep{schmudgen1991k} is constructed by choosing SOS polynomials as the base class, and using sums of the product of these SOS polynomials with products of the $g_i$'s~(see \eqref{eq:schmu}). By further assuming that $S_g$ is a polyhedron, a {\em simpler} \german{}; that is,  Handelman's \german{}~\citep{Hand88}, can be derived in which the SOS polynomials are replaced (in  \schmu{}'s \german{}) by nonnegative constants (see, Proposition~\ref{prop:HandNneg}).

The key element of our approach is to reduce the problem of checking nonnegativity over~$S_g$ to the problem of checking nonnegativity over a simpler  set of higher dimensions. Namely, to certify the nonnegativity of~$p$ on~$S_g$, it is enough to write~$p$ in the form $p(x) = F(x, g_1(x), \dots, g_m(x))$ where $F \in \R[x,u]$ is nonnegative on some \emph{simple} set~$T\subset \R^{n+m}$ satisfying
\begin{equation}
\label{eq:hatSg}
\hat S_g :=  \{(x,g_1(x),\dots,g_m(x)): x \in S_g\} \subseteq T \subset  \R^n\times \R^m_+.
\end{equation}
We call such $F$ a {\em lifting} of $p$ to $T$.

The crux of our work is to show that liftings, with {\em simpler} choices of the set $T$, exist for polynomials positive on general semialgebraic sets under appropriate conditions.
More specifically, when~$S_g$ is compact, we show (Theorem~\ref{thm:mainCompGen}) that given any compact set~$T$ satisfying $\hat S_g\subset T \subset \R^n\times R^m_+$, there exists a lifting of~$p$ to~$T$.
Notice that if $S_g$ is compact, then~$\hat S_g$ is also compact and simple sets $T$ satisfying condition~\eqref{eq:hatSg} always exist. For instance, we can choose $T$ to be a box. This ample freedom to choose $T$ allows for the design of NNCs fitted for different purposes, that is,
NNCs with varying sparsity structures,  using different base classes,
and in particular, not necessarily SOS polynomials (see Figure~\ref{fig:polys}).
We illustrate this by considering two possible choices of $T$, a box and a generalized simplex.

We extend the proposed approach to unbounded sets $S_g$, by choosing $T$ to be the nonnegative orthant in higher dimensions.
However,  $g \in \R[x]^m$ may be such that the nonnegative orthant does not contain $\hat S_g$.
In such cases, we make a two-step reduction.
First, we embed $S_g \subset \R^n$ on $S_{g'} \subset \R^{2n}_+$ via $x \mapsto (x^-,x^+)$ where $x^- = -\min(x,0)$ and $x^+ = \max(0,x)$ and $g'(x^+, x^-):= g(x^+ - x^-)$. Then, we use that $S_{g'} \subset \R^{2n}_+$ implies $T := \R^{2n+m}_+ \supset \hat S_{g'} $.

An essential property of our liftings is that their degree is known a priori (see Theorems~\ref{thm:mainCompGen} and~\ref{thm:ineqGen} for details).
This degree is bounded by the maximum among the degree of $p$ and twice the degrees of the polynomials $g_j$, $j=1,\dots,m$. This fact allows us to work in the finite-dimensional vector space of fixed-degree polynomials. Thus, our proofs contrasts with the proofs of other well-known \germans{}, such as \cite{Puti93,schmudgen1991k}, which require infinite-dimensional algebraic geometry tools. Also, our lifting strategy fundamentally differs from the approaches that use the  general \german{} by Marshall \citep[][Thm. 5.4.4]{Marshall2008}, such as the SONC \german{} in~\citep[][Thm. 4.8]{dressler2017positivstellensatz}, and the
 {\em sums of arithmetic-mean/geometric-means} (SAG) \german{} in~\citep[][Thm. 4.2]{chandrasekaran2016relative}.  Our liftings for unbounded sets are also distinct from   homogenization-based techniques for unbounded sets, such as  \cite{putinar1999solving,dickinson2015extension, mai2021positivity}, where the primary approach is to ``compactify" unbounded sets via homogenization.

To evidence the impact of the proposed methodology, we study relevant applications of our approach.

\subsection{Applications} \label{subsec:intro_summary}

Most existing \germans{} for polynomials over semialgebraic sets derive from algebraic geometry and are not created with optimization applications in mind. NNCs over the semialgebraic set~$S_g$ are classically constructed as rational expressions whose numerator and denominator are sums of products of $g_j$, $j=1,\dots,m$ and polynomials taken from a set of nonnegative polynomials~\citep[see, e.g.,][]{artin, habicht1939zerlegung, schmudgen1991k, Puti93, HardLP88, Hand88}.  As mentioned earlier, we call the latter set of polynomials the base class of the NNC and refer to the NNC as one {\em based} on that class of polynomials.
SOS polynomials are the most common base class~\citep[see, e.g.,][]{artin, habicht1939zerlegung, schmudgen1991k, Puti93}. Thus, the latter articles provide \germans{} based  on SOS polynomials, or {\em SOS \germans{}} for brevity.
In more generality, throughout the article, we will use $\K$ to refer to potential base classes (e.g., SOS, SONC, and SDSOS polynomials, among others).
To be amenable to optimization, NNCs must have as a denominator a fixed (known) polynomial~\citep[see, e.g.,][]{habicht1939zerlegung, schmudgen1991k, Puti93, HardLP88, Hand88} and the membership to the base class must be efficiently representable.
For instance, one can test if a polynomial is an SOS of a known degree using {\em semidefinite programming} (SDP)~\citep[see, e.g.,][]{BlekPT13}. Thus, the degree-restricted version of SOS \germans{} (i.e., when the degree of the SOS polynomials in the associated NNC is restricted to be at most a given degree) leads to {\em linear matrix inequality} (LMI) approximation hierarchies for PO problems that can be solved in polynomial time, up to a given accuracy, using interior-point methods~\citep[see, e.g.,][]{NestT98}.

\subsubsection*{Non-SOS Schm\"{u}dgen-type \germans{} over compact sets}

Solving SDPs is computationally expensive~\citep[see, e.g.,][]{lasserre2006convergent}.
One alternative that has been recently explored to address the computational effort required to solve SDPs associated with SOS \germans{} is to derive \germans{} in which the base class is not the SOS polynomials; that is, {\em non-SOS} \germans{}~\citep[see, e.g.,][]{dressler2017positivstellensatz, dickinson2015extension, chandrasekaran2016relative}. Some of these results can be referred to as {\em non-SOS Schm\"{u}dgen-type} \germans{}, as they follow the structure of the fundamental
Schm\"{u}dgen \german{}~\cite{schmudgen1991k}.
In particular, after adding appropriate redundant constraints to the definition of $S_g$, non-SOS Schm\"{u}dgen-type \germans{} have been derived based on the classes of  SONC~\citep{dressler2017positivstellensatz}, SAG~\citep[][Thm. 4.2]{chandrasekaran2016relative}, and {\em hyperbolic}~\citep{hyperb1} polynomials.
Also, non-SOS Schm\"{u}dgen-type NNCs based on {\em diagonally dominant SOS} (DSOS) and SDSOS polynomials provide reasonable bounds in some practical applications~\citep{ahmadi2019dsos}.

\begin{wrapfigure}{r}{0.35\textwidth}
%\begin{figure}[H]
    \centering
    %\fbox{
    \begin{adjustbox}{width=0.30\textwidth}
    \begin{tikzpicture}

\begin{scope}[shift={(3cm,-5cm)}, fill opacity=0.5,
  mytext/.style={text opacity=1,font=\large}]

\draw[draw = black] (0,0) circle (5);
\draw[fill = gray!20, draw = black,name path=circle 1] (-1.5,0) circle (3);
\draw[fill = gray!35, draw = black,name path=circle 2] (1.5,0) circle (3);
\draw[fill = gray!75, draw = black,name path=circle 3] (1.3,0) arc(0:180:1.3 and 1);
\draw[fill = gray!75, draw = black,name path=circle 4] (-1.3,0) arc(-180:0:1.3 and 2.2);
\draw[fill = gray!95, draw = black,name path=circle 5] (0,-1) circle (0.9);
%\draw[fill = gray!80, draw = black,name path=circle 3] (0,0) ellipse (1 and 2);

\pgftransparencygroup
\clip (-1.5,0) circle (3);
\fill[gray!50] (1.5,0) circle (3);
%\filldraw[draw,fill=green,name intersections={of=circle 1 and circle 2}] (intersection-1) .. controls +(-4,1) and +(-4,-1) ..(intersection-2);
\endpgftransparencygroup

\node[mytext] at (0,4) (A) {\bf Non-negative};
\node[mytext] at (-2.8,0.4) (B) {\begin{tabular}{c} \bf SOS \\ (SDP) \end{tabular}};
\node[mytext] at (2.8,0.4) (C) {\begin{tabular}{c} \bf SONC \\ (SOCP, GP, \\ REP) \end{tabular}};
\node[mytext] at (0,1.5) (D) {\begin{tabular}{c} \bf SDSOS \\ (SOCP) \end{tabular}};
\node[mytext] at (0,0.4) (E) {\begin{tabular}{c} \bf DSOS \\ (LP) \end{tabular}};
\node[mytext] at (0,-0.8) (E) {\begin{tabular}{c} $\mathbf{\mathbb{R}_+}$ \\ (LP) \end{tabular}};
\end{scope}
\end{tikzpicture}
\end{adjustbox}
   % \includegraphics[width = 0.26\textwidth]{figpolys}
    %}
    \caption{Nonnegative polynomial classes. Related optimization model is shown in parentheses. \label{fig:polys}}
%\end{figure}
\end{wrapfigure}
Our lifting approach leads to general non-SOS Schm\"{u}dgen-type \germans{} (Proposition~\ref{prop:genHand}) in which any
class $\K$ that contains the nonnegative constants (i.e., $\R_+ \subseteq \K$) could be used as the base class.
Our  non-SOS Schm\"{u}dgen-type \germans{} answer an open question underlying the work in~\citep{dressler2017positivstellensatz, chandrasekaran2016relative}; namely, for which base classes~$\K$, there exists a~$\K$ Schm\"{u}dgen-type \german{} on compact sets?
In particular, from Figure~\ref{fig:polys}, it is easy to see that the SONC \german{}~\citep{dressler2017positivstellensatz} and SAG \german{}~\citep[][Thm. 4.2]{chandrasekaran2016relative} readily follow from our result. A wealth of classes of polynomials (see Figure~\ref{fig:polys}) can be used as base classes in the~\germans{}.
Interestingly, our result also shows how to use DSOS/SDSOS polynomials  to construct NNCs which are guaranteed to exist over general compact semialgebraic sets (Corollary~\ref{cor:SDSOS}).
This is in contrast with the aforementioned DSOS/SDSOS NNCs proposed in~\citep{ahmadi2019dsos} for numerical purposes,
which do not always exist~(see, \cite{josz2017counterexample,ahmadi2017response} and Example~\ref{ex:sdsos}).

\looseness -1
When using our lifting method to obtain non-SOS Schm\"{u}dgen-type \germans{} via the reduction to the box (Proposition~\ref{prop:genHand}), redundant lower and upper bound box constraints for $S_g$ are introduced. That is, constraints of the form $x_i - L_i$ and $U_i - x_i$, $i=1,\dots,n$, appear explicitly in the certificates (see~\eqref{eq:genHand}). The inclusion of such constraints might seem like a drawback of our method.
However, we show (Section~\ref{subsubsec:SDSOS_bounds}) that our lifting method adds the right level of complexity to the NNC, to allow the use of non-SOS base classes; that is, this type of non-SOS NNCs may not exist on compact semialgebraic sets unless such redundant constraints are included. In particular, our non-SOS Schm\"{u}dgen-type \germans{} provide an answer to another open question underlying~\citep{dressler2017positivstellensatz, chandrasekaran2016relative}; namely, what types of terms are needed to obtain non-SOS Schm\"{u}dgen-type \germans{}?

\subsubsection*{Sparse SOS \germans{}  over compact sets}
\label{sec:sparseapp}
A few years after the introduction of the Lasserre hierarchy~\cite{lasserre2001global}, which uses SOS polynomials as the base class, sparsity was exploited to improve the numerical scalability of this hierarchy~\cite{waki2006sums}. This method was implemented in \cite{waki2008algorithm} and proven to converge in~\cite{lasserre2006convergent}.
Since then, deriving structured variations of classical \germans{} has become an important avenue  to make SOS \germans{} more amenable to optimization.  In particular, sparsity~\citep[see, e.g.,][]{wang2019tssos, wang2020chordal,sparse20,Weisser2018,magron2022sparse} and symmetry~\citep[see, e.g.,][]{riener2013exploiting, gatermann2004symmetry} are exploited in order to reduce the size of the corresponding optimization problem obtained when using (degree-restricted versions of) such certificates in~\eqref{eq:NNC-OPT}.

Along these lines, our lifting approach leads to the derivation of a {\em semi-sparse} SOS \german{} (Proposition~\ref{prop:sparse}) whose associated NNC  has some sparsity structure, even if the underlying semialgebraic set does not possess any sparsity properties that are  beneficial for the existing techniques~\citep{waki2006sums,waki2008algorithm,lasserre2006convergent,wang2019tssos, wang2020chordal,sparse20,Weisser2018,magron2022sparse}. It is worth mentioning that to obtain the latter result using our general lifting strategy, we first derive a new sparse \german{} over the simplex  (Lemma~\ref{lem:sparseDelta}) which is of independent interest.

\subsubsection*{Non-SOS \germans{} for unbounded sets}

We contribute to the growing literature (see, Section~\ref{sec:copocert} and ~\citep{mai2021positivity}) on NNCs that are guaranteed to exist even if the underlying semialgebraic set is not compact.
In the unbounded case, large families of semialgebraic sets where Putinar's certificate definitely cannot be constructed are known~\cite{cimprivc2011}. Also, non-SOS \germans{} for unbounded sets exist only for very simple sets, such as P{\'o}lya's \german{}~\citep[][Sec. 2.2]{HardLP88} over the nonnegative orthant, which is based on nonnegative constants. Thanks to the lifting for unbounded sets (Theorem~\ref{thm:ineqGen}), we extend P{\'o}lya's \german{} to obtain \germans{} for polynomials on generic semialgebraic sets (Proposition~\ref{prop:genPolya}). Beyond using SOS polynomials, this allows the use of different base classes $\K$ such as SONC, SAG, and DSOS/SDSOS polynomials, to certify the nonnegativity over unbounded semialgebraic sets. This option  has not been explored in the literature.

A typical condition needed to guarantee the existence of an NNC for a polynomial on a compact semialgebraic set is that the polynomial must be \emph{positive} over the set (see e.g., Theorem~\ref{prop:genHand}).
Not surprisingly, one needs a bit more in the case of unbounded sets (see Example~\ref{ex:simpleex}). We show that in the latter case it is  sufficient for the polynomial to be {\em strongly positive} (see Definition~\ref{def:posInf}) over the set.
 Further, we show that this condition holds {\em generically}. That is, for almost every semialgebraic set and positive polynomial over the set  (Corollary~\ref{cor:generic}). Genericity is formally defined in Definition~\ref{def:genericity} following \citep{Nie2013, WangRational, lasserre2009convex}.

\looseness -1
It is important to mention that the lifting approach, the non-SOS Schm\"{u}dgen-type \germans{} (Proposition~\ref{prop:genHand}), and the~\germans{} terms analysis of Section~\ref{subsubsec:SDSOS_bounds},  turn out to be the stepping stone to positively answer an open question posed in~\citep[][Sec. 6]{dressler2017positivstellensatz}; namely, whether there exist {\em non-SOS Putinar-type} \germans{} which follow the structure of the fundamental Putinar's \german{}~\citep{Puti93}, but are based on a non-SOS class~\citep[see,][]{roebers2021sparse}.

Besides providing relevant examples, we illustrate our results by performing two sets of computational experiments. In Section~\ref{sec:SchNumerics}, we benchmark the lower bounds obtained for PO problems using the proposed non-SOS \schmu-type \german{} (Proposition~\ref{prop:genHand}) versus the ones obtained using Putinar's NNC~\eqref{eq:puti} with different base classes.
In particular, we compare our approach and approaches proposed in~\cite{lasserre2001global, ahmadi2019dsos} to compute the lower bounds. Further,  in Section~\ref{sec:sparseNumerics}, we show the advantages of the proposed {\em semi-sparse} SOS \german{} (Proposition~\ref{prop:sparse}) for PO problems which do not have any sparsity structure that can be easily exploited with the existing methodologies for sparse PO problems in the literature~\citep[e.g.,][]{waki2006sums,waki2008algorithm,lasserre2006convergent,wang2019tssos, wang2020chordal,sparse20,Weisser2018,magron2022sparse}.

\looseness -1
Additional computational experiments are presented in Examples~\ref{ex:unbounded_mp} and~\ref{ex:unbounded_new} to illustrate the methodology presented for unbounded sets.
Example~\ref{ex:unbounded_mp} illustrates how to use the P{\'o}lya's \german{} extension (Proposition~\ref{prop:genPolya}) to solve PO problems with unbounded feasible sets, and in particular, PO problems
for which  hierarchies based on  Putinar's \german{}~\eqref{eq:puti} fail to give a lower bound for the problem.
Example~\ref{ex:unbounded_new} illustrates that  the P{\'o}lya's \german{} extension (Proposition~\ref{prop:genPolya}) can be used to
obtain or closely approximate the optimal value of PO problems with unbounded feasible set using non-SOS classes of polynomials as the base class.

All the computational experiments mentioned above were done using MATLAB R2021a, Yalmip~\cite{Lofberg2004}, and some functions from SOSTOOLS~\citep{sostools} and DIGS~\citep{ghaddar2016dynamic}, on a computer with  processor Intel\textsuperscript{\textregistered} Core\textsuperscript{\texttrademark} i7-8665U CPU @ 1.90GHz and 16 GB of RAM. Semidefinite programs are solved with MOSEK, Version 9.3.20 \cite{mosek}, and linear programs are solved with Gurobi, Version 10.0.1  \cite{gurobi}. All the data and code used to generate the computational results presented in the article are publicly available on the Github webpage \url{https://github.com/OlgaKuryatnikova/PolyLift}.

\section{Preliminaries}
\label{sec:prelim}

We denote by $\R[x]:=\R[x_1,\dots,x_n]$ the set of $n$-variate polynomials with real coefficients, and by $\R_d[x]$ (respectively $\R_{=d}[x]$) the subset of $\R[x]$ of  polynomials of degree not larger than (resp. equal to)~$d$.
For any $d \in \N$ and $\alpha \in \N^n$, let
${d \choose \alpha}$ denote the multinomial coefficient  ${d \choose \alpha} := \frac{d!}{(d-e\tr \alpha)!\alpha_1!\cdots \alpha_d!}$ and $\N^{n}_d := \{\alpha \in \N^{n}: e\tr \alpha \le d\}$, where $e$ denotes the vector of all-ones in appropriate dimension.
Then, given~$p(x) \in \R[x]$ with $\deg p \le d$, we can write
\begin{equation}
\label{eq:polysum}
p(x) = \sum_{\alpha \in \N^n_d} {d \choose \alpha} p_\alpha x^\alpha,
\end{equation}
for some $p_\alpha \in \R$, where $x^\alpha := x_1^{\alpha_1} \cdots x_n^{\alpha_n}$, for all
$\alpha \in \N^{n}_d$. We define $\|p\| = \max \{|p_\alpha|:  \alpha \in \N^{n}_d\}$ and for any $x \in \R^n$, let $|x| \in \R^n$ denote the component-wise absolute value of $x$. In particular, for any $\alpha \in \N^n$, $|x|^{\alpha} = |x^{\alpha}|$.

The following bound on the value of a polynomial will be useful throughout the article.

\begin{lemma}\label{lem:normIneqPol}
Let $p \in \R[x]$. For any $x \in \R^n$ we have
\[|p(x)| \le \|p\|(1+e \tr |x |)^{\deg p}.\]
\end{lemma}
\begin{proof}
Given $p \in \R[x]$ with $\deg p = d$, and $x \in \R^n$ we have
\begin{align*}
|p(x)| &\le \sum_{\alpha \in \N^n_d} {d \choose \alpha}|p_\alpha ||x|^\alpha \le \|p\| \sum_{\alpha \in \N^n_d} {d \choose \alpha}|x|^\alpha   =  \|p\| (1+e \tr |x |)^d. \end{align*}
\end{proof}

When appropriate, we will use the following notation for ease of presentation. Given  $g_1,\dots,g_m \in \R[x]$ and $\alpha \in \N^{m}_d$ we use $g(x)^{\alpha} := \prod_{j=1}^m g_j(x)^{\alpha_j}$ (in particular, $x^\alpha = \prod_{j=1}^m x_j^{\alpha_j}$). Also, we use $g \in \R[x]^m$  to denote the array $g := [g_1,\dots,g_m]\tr$ of polynomials.

For any $S \subseteq \R^n$, we define
\[
\P(S) = \{ p \in \R[x]: p(x) \ge 0 \text{ for all } x \in S\},
\]
as the set of polynomials nonnegative on $S$. Similarly, we define
\[
\P^+(S) = \{ p \in \R[x]: p(x) > 0 \text{ for all } x \in S\},
\]
as the set of polynomials positive on $S$.
Furthermore, let $\P_d(S):=\P(S)\cap\R_d[x]$
(resp. $\P^+_d(S):=\P^+(S)\cap\R_d[x]$)
denote the set of polynomials of degree at most~$d$ that are nonnegative
(resp. positive)
on $S$. As mentioned in the introduction, given~$g \in \R^m[x]$, we denote by $S_g \in \R^n$ the semialgebraic set defined by~$g$; that is, $S_g = \{x \in \R^n: g_1(x) \ge 0, \dots, g_m(x) \ge 0\}$.

We frequently refer to a fundamental class of nonnegative polynomials, namely, the {\em sum-of-squares} (SOS) polynomials~\citep{BlekPT13}.  A polynomial $p \in \R_{2d}[x]$ is SOS if $p(x) = \sum_{i\le l} q^i(x)^2$ for some $q^1,\dots,q^l \in \R_d[x], l \in \N$.

Central to our discussion are two simple  sets. The first set is a \emph{box} in~$\R^n$. Given $L, U \in \R^n$ define
\[
\Box^n_{L,U}:=\{x\in \R^n: L\le x\le U\}.
\]
The second set is a \emph{generalized simplex} in~$\R^n$. Given $M \in \R$ and $L \in \R^n$, define
\[
\Delta^n_{L,M}:=\{x\in \R^n: x\ge L, \ e\tr x\le M\}.
\]
 Note that $\Delta^n_{\mathbf{0},1}$ is the {\em simplex}. For brevity, we will abuse the notation by referring to $\Delta^n_{L,M}$ as a simplex.
Through the article, we will use the following bounds on the value of a polynomial over these simple  sets.

\begin{lemma}\label{lem:hatUR} Let $p \in \R[x]$. Then, the following holds.
\begin{enumerate}[label=(\alph*)]
\item Box case: Let $L,U \in \R^n$ such that $L \le U$ be given. Then, we have that
\[
\displaystyle \max_x \left \{|p(x)|: L \le x \le U \right \}
\leq \|p\|(1+e\tr B)^{\deg p},
\]
where $B_i=\max \{|U_i|,|L_i|\}$ for $i\in\{1,\dots,n\}$.  \label{lem:hatUR_b}
\item Simplex case: Let $L,M \in \R^n$ such that $L \le M$ be given. Then, we have that
\[
\displaystyle \max_x \left  \{|p(x)|: x\ge L, e\tr x \le M \right \} \leq  \|p\|(1+M + e\tr(|L|-L))^{\deg p}.
\]
\label{lem:hatUR_c}
\end{enumerate}
\end{lemma}
\begin{proof}  Statement \ref{lem:hatUR_b} follows from Lemma~\ref{lem:normIneqPol}. Next, we prove statement  \ref{lem:hatUR_c}.
 Let $x\in \R^n$ be such that $L \le x$ and $e\tr x \le M$. Then $e\tr |x| \le e\tr(x-L) + e\tr |L| \le M + e\tr(|L|-L)$. From Lemma~\ref{lem:normIneqPol} we obtain that
\[
|p(x)| \le \|p\|(1+e\tr|x|)^{\deg p} \le \|p\|(1+M + e\tr(|L|-L))^{\deg p}.
\]
\end{proof}

\section{The compact case}
\label{sec:compactcase}

Extensive work in the area of PO, starting with the seminal work of~~\citet{lasserre2001global}, has shown how SOS \germans{}, such as \schmu's \german{}~\citep{schmudgen1991k} and Putinar's \german{}~\citep{Puti93}, can be used as a fundamental tool to develop solution methodologies for PO problems with a compact semialgebraic feasible set. One of the drawbacks of using these SOS \germans{} is that checking membership in the class of SOS polynomials of a fixed degree requires the solution of an SDP; a task that in practical terms is computationally expensive~\citep[see, e.g.,][]{lasserre2006convergent}. As a result, recent literature has focused on developing NNCs that either use alternative base classes~\citep[see, e.g.,][]{dressler2017positivstellensatz, chandrasekaran2016relative} or that have special structure~\citep[see, e.g.,][]{lasserre2006convergent,wang2019tssos}, to make them more amenable, optimization-wise, to be used in solution methodologies for PO problems. In what follows, we
develop a method to obtain new such certificates.

\subsection{Main result}
We begin by presenting the following key Proposition on certifying the nonnegativity of a polynomial over a compact set intersected with a semialgebraic set defined by equality constraints; that is, a {\em variety} of the form $V=\{x\in \R^n:  h_1(x) = 0,\dots,  h_m(x) = 0\}$, where $h \in \R^m[x]$.

\begin{proposition}[{\citep[][Cor. 2]{PenaVZ08}}]   \label{prop:CompEqu} Let $S \subset \R^n$   be a non-empty compact set, and let $h_1,\dots,h_m,p \in \R[x]$ be such that $h_j\in \P(S)$ for  $j = 1,\dots,m$. Define $d_{\max}=\max \left \{\deg h_1,\dots,\deg h_m, \deg p \right \}$ and $V=\{x\in \R^n:  h_1(x) = 0,\dots,  h_m(x) = 0\}$. Then  $p\in  \P^+(S\cap V)$ if and only if there are $F \in \P^+_{d_{\max}}(S)$ and $\alpha_j\in \R_{d_{\max}-\deg h_j}[x]$ for  $j = 1,\dots,m$ such that
\[p(x) = F(x)+\sum_{j=1}^m \alpha_j(x) h_j(x).\]
\end{proposition}

In our context, $F$ in Proposition~\ref{prop:CompEqu} can be interpreted as a lifting for $p$ from the intersection of a compact set $S$ and a variety $V$ to $S$. Namely, Proposition~\ref{prop:CompEqu} reduces the problem of certifying the nonnegativity of~$p$ over the {\em complex} set $S \cap V$, to the problem of certifying nonnegativity of~$F$ over the {\em simpler } set $S$.
Next, we use Proposition~\ref{prop:CompEqu} to prove our main result on the compact case. Namely, given $p \in \R[x]$, $g \in \R^m[x]$ such that $S_g$ is compact, and $T \subset \R^{n} \times \R^m_+$ such that condition~\eqref{eq:hatSg} holds, we use the proposition to construct a lifting $F \in \P^+(T)$ of~$p$. We use such liftings to construct certificates of nonnegativity over~$S_g$ from certificates of nonnegativity over~$T$, as explained in Section~\ref{sec:approach}.

\begin{theorem}
\label{thm:mainCompGen}
Let $p \in \R[x]$, and $g_1,\dots,g_{m} \in \R[x]$ be such that  $S=\{x\in \R^n: g_1(x)\ge 0,\dots,g_m(x) \ge 0\}$ is a non-empty compact set. Consider another compact set $T \subset \R^{n} \times \R^m_+$ such that $\left \{(x,g(x)):x \in S \right \} \subseteq T$, and define $d_{\max}=\max \left \{2\deg g_1,\dots,2\deg g_m,\deg p  \right \}$. Then, $p\in \P^+(S)$ if and only if there exists $F \in \P^+_{d_{\max}}(T)$ such that
\begin{equation}
\label{eq:compactlifting}
p(x) = F(x,g(x)).
\end{equation}
\end{theorem}
\begin{proof}
In one direction, note that $F \in \P^+_{d_{\max}}(T)$,  $\left \{(x,g(x)):x \in S \right \} \subseteq T$, and~\eqref{eq:compactlifting}, readily imply that $p\in \P^+(S)$. For the other direction,
let $d_j = \deg g_j$, $j=1,\dots,m$. Define $\hat g_j: \R^{n+m} \to \R$ as $\hat g_j(x,u): = (g_j(x)-u_j)^2$ for $j = 1,\dots,m$. Now, define  \[U:=\left \{(x,u) \in T: \hat g_j(x,u)=0 \text{ for } j = 1,\dots,m \right \}.\]
Let $q(x,u):=p(x)\in \R_{d_{\max}}[x,u]$. By definition of $\hat g_1,\dots,\hat g_m$, we have that $\hat g_j \in \P(T) \text{ for } j = 1,\dots,m $. From the definition of  $U$, it follows that $q\in \P_{d_{\max}}^+(U)$.
Then, Proposition~\ref{prop:CompEqu} implies that there is $F \in \P^+_{d_{\max}}(T)$ and  $\alpha_j \in \R_{d_{\max}-2d_j}[x,u]$,  for all $j = 1,\dots,m$ such that
\begin{align*}
p(x) = q(x,u) = F(x,u) +  \sum_{j=1}^m\alpha_j(x,u) \hat g_j(x,u).
\end{align*}
Replace $u \leftarrow g(x)$  to obtain
$
p(x) = q(x,g(x)) = F(x,g(x)).
$
\end{proof}

Notice that the expression~\eqref{eq:compactlifting} allows us to reduce the problem of certifying the nonnegativity of $p$ on $S_g \subset \R^n$, to the problem of certifying the nonnegativity of a polynomial $F$ over an appropriate higher-dimensional set $T \subset \R^n \times \R^m_+$. We use the term reduce as there is freedom in choosing $T$. In particular, one might pick $T$ as a simpler  set than $S_g$, taking advantage of \germans{} valid on $T$ to certify the nonnegativity of $p$ on $S_g$.
 Next, we specialize Theorem~\ref{thm:mainCompGen} to the case when $T$ is a box or a simplex.

\begin{corollary}\label{cor:mainBasicCompact} Let $p \in \R[x]$, and $g_1,\dots,g_{m} \in \R[x]$ be such that  $S=\{x\in \R^n: g_1(x)\ge 0,\dots,g_m(x) \ge 0\}$ is a {non-empty} compact set. Define $d_{\max}= \allowbreak \max \{2\deg g_1, \allowbreak \dots, \allowbreak 2\deg g_m, \allowbreak \deg p\}$.
\begin{enumerate}[label=(\alph*)]
  \item \textbf{Box:} Let $L,U \in \R^n$ be such that  $S \subseteq \Box^n_{L,U}$. For $j\in \{1,\dots,m\},$ let
  $U^g_j \ge \|g_j\|(1+e\tr B)^{\deg g_j}$, where $B_i \ge \max \{|U_i|,|L_i|\}$ for $i\in\{1,\dots,n\}$. Define $\hat U:= [U\tr, (U^g)\tr]\tr$, $\hat L:= [L\tr,\mathbf{0}\tr]\tr$.  Then, $p\in \P^+(S)$ if and only if  there exists $F \in \P^+_{d_{\max}}(\Box^{n+m}_{\hat U, \hat L})$ such that \label{cor:mainBasicCompactBox}
$p(x) = F(x,g(x))$.
\label{item:boxcert}

\item \textbf{Simplex:} Let $L \in \R^m$ and $M \in \R$ be such that  $S \subseteq \Delta^n_{L,M}$, and let $\hat M \ge M+\sum_{j=1}^m (1+M + e\tr(|L|-L))^{\deg g_j}\|g_j\|$, $\hat{L}:=[L\tr, \boldsymbol{0}\tr]\tr$. Then $p\in \P^+(S)$ if and only if  there exists $F \in \P^+_{d_{\max}}(\Delta^{n+m}_{\hat L, \hat M})$ such that \label{cor:mainBasicCompactSimplex}
$p(x) = F(x,g(x))$.
\label{item:simplexcert}
\end{enumerate}

\begin{proof}
In case~\ref{item:boxcert} (resp. case~\ref{item:simplexcert}), Lemma~\ref{lem:hatUR}\ref{lem:hatUR_b} (resp. Lemma~\ref{lem:hatUR}\ref{lem:hatUR_c}) implies that the set $T= \Box^{n+m}_{\hat U, \hat L}$ (resp. $T=\Delta^{n+m}_{\hat L, \hat M}$) satisfies the conditions of Theorem~\ref{thm:mainCompGen}. Thus, the statement follows from Theorem~\ref{thm:mainCompGen}.
\end{proof}

\begin{remark}
Lemma~\ref{lem:hatUR} is instrumental in proving Corollary~\ref{cor:mainBasicCompact}. It also implies that the representation of $p$ obtained in Corollary~\ref{cor:mainBasicCompact} is indeed an NNC for $p$ on $S$.
\end{remark}
\end{corollary}

Corollary~\ref{cor:mainBasicCompact} allows translating NNCs known for simple  sets, such as a box or a simplex, to NNCs over general \emph{compact} sets.
Next, we illustrate the potential of this result to obtain new NNCs. We use the box as the simple set in  Section~\ref{sec:SchGen}, and the simplex as the simple set in Section~\ref{subsec:sparse}.

\subsection{Non-SOS Schm\"{u}dgen-type \germans{}}
\label{sec:SchGen}
  \schmu's \german{}~\citep{schmudgen1991k} states that given $g_1, \dots, g_m \in \R[x]$ such that $S_g$ is compact, and  $p \in \P^+(S_g)$, there exist SOS polynomials $\sigma_{\alpha}$, $\alpha \in \{0,1\}^m$, such that
\begin{equation}
\label{eq:schmu}
p(x) = \sum_{\alpha \in \{0,1\}^m} \sigma_{\alpha}(x)g(x)^{\alpha}.
\end{equation}
Since $g_j(x) \ge 0$ for any $x \in S_g$, $j=1,\dots,m$ and $\sigma_{\alpha}(x) \ge 0$ for any $x \in \R^n$, $\alpha \in \{0,1\}^m$, the expression~\eqref{eq:schmu} is an NNC; that is, it shows that $p \in \P(S_g)$. The base class in this NNC is the SOS polynomials.

As mentioned, using SOS in optimization is computationally expensive since it usually requires solving SDPs. Hence, it is a topic of great interest to develop alternatives to SOS \germans{}, such as \eqref{eq:schmu}, by replacing the SOS polynomials with other base classes. In fact,
\germans{} following the form of~\eqref{eq:schmu} have been obtained using SONC~\citep{dressler2017positivstellensatz}, SAG~\cite[][Thm. 4.2]{chandrasekaran2016relative}, and hyperbolic~\cite{hyperb1} polynomials as base classes. We refer to these results as \emph{non-SOS \schmu-type \germans{}}.

Next,
departing from Handelman's \german{}~\cite{Hand88} (see Proposition~\ref{prop:HandNneg}) to certify nonnegativity over a box,
we use the lifting to a box in Corollary~\ref{cor:mainBasicCompact}\ref{item:boxcert} to derive a {\em general} non-SOS \schmu-type \german{} (Proposition~\ref{prop:genHand}); that is,
unlike the results in~\citep{dressler2017positivstellensatz, chandrasekaran2016relative} this \german{} holds for more general base classes $\K$. Specifically, any class of polynomials $\K$ containing the set of nonnegative constants
($\R_+ \subseteq \K$) fits as the base class.
As a result, our
 non-SOS Schm\"{u}dgen-type \germans{} provide an answer to an open question underlying the work in~\citep{dressler2017positivstellensatz, chandrasekaran2016relative}; namely, for which base classes~$\K$, there exists a~$\K$ Schm\"{u}dgen-type \germans{} can be proved on a compact set $S_g$.
Since  nonnegative constants are SONC and SAG polynomials (see Figure~\ref{fig:polys}),
Proposition~\ref{prop:genHand}  implies the SONC \german{} derived in~\cite[][Thm. 4.8]{dressler2017positivstellensatz} and the SAG \german{} derived in~\cite[][Thm. 4.2]{chandrasekaran2016relative}. Further, nonnegative constants are DSOS and SDSOS polynomials (see Figure~\ref{fig:polys}), a fact that is used in Corollary~\ref{cor:SDSOS} to derive new \germans{} based on these classes of polynomials.

An advantage of the \german{} in Proposition~\ref{prop:genHand} is its flexibility to derive NNCs with different characteristics due to the fact that there is a lot of freedom in choosing the base class~$\K$. In fact,~$\K$ does not have to be a class of nonnegative polynomials (see Remark~\ref{rem:base_class}), as $\K$ only needs to contain the set of nonnegative constants.
Note that membership in the class of SONC polynomials of fixed degree can be tested using {\em geometric programming} (GP)~\citep{iliman2016lower, papp2019duality}, {\em second-order cone programming} (SOCP)~\citep{wang2020second}, or {\em relative entropy programming} (REP)~\citep{dressler2017positivstellensatz}. Also,
membership in the class of SAG polynomials of fixed degree can be tested using REP~\citep{chandrasekaran2016relative, karaca2017repop} and geometric programming~\citep{ghasemi2012lower}. Further, membership is the class of SDSOS (resp. DSOS) polynomials of fixed degree can be tested using second-order cone (resp. linear) programming.
Consequently, our general non-SOS \german{} provides a variety of computational alternatives to using SDP for solving PO problems (see Section~\ref{sec:SchNumerics}). In what follows, we formally derive this result.

\begin{proposition}[Handelman's \german~\citep{Hand88}] \label{prop:HandNneg}
Let $A \in \R^{m \times n}$, $b \in \R^m$ be such that $S = \{x \in \R^n:Ax \le b\}$ is a non-empty polytope. If $p\in \P^+(S)$, then there exists
$M \in \N$ and $c_{\alpha} \ge 0$ for all $\alpha \in \N^{m}_M$ such that
\begin{align*}
p(x) =
 \sum_{\alpha \in \N^{m}_M} c_{\alpha}(b-Ax)^{\alpha}.
\end{align*}
\end{proposition}
An explicit value of the bound $M$ in Proposition~\ref{prop:HandNneg} is provided in~\citep[][Thm.~3]{PoweR01}.

We obtain the desired result below using Proposition~\ref{prop:HandNneg} and Corollary~\ref{cor:mainBasicCompact}\ref{item:boxcert} (i.e., setting $T$ as the box for the lifting in Theorem~\ref{thm:mainCompGen}).

\begin{proposition}[Non-SOS \schmu-type \german{}]
\label{prop:genHand}
Let a base class $\K$ satisfying  $\R_+ \subseteq \K \subset \R[x] $ be given. Let $g_1,\dots,g_m \in \R[x]$ be such that $S=\{x\in \R^n:  g_1(x)\ge 0,\dots,g_m(x)\ge 0\}$ is a non-empty compact set. Let $L,U \in \R^n$ be such that  $ S \subseteq \{x \in \R^n:  L \le x \le U\}$. Given $p\in \P^+(S)$,
 there exist $r \ge 0$ and $c_{\alpha, \beta, \gamma} \in \K$ for $(\alpha, \beta,\gamma) \in \N^{2n+m}_{r}$ such that
\begin{equation}
p(x) =
 \sum_{(\alpha, \beta,\gamma ) \in \N^{2n+m}_{r} } c_{\alpha, \beta,\gamma}(x) (x-L)^{\alpha} (U-x)^{\beta}g(x)^\gamma. \label{eq:genHand}
\end{equation}
\end{proposition}
\begin{proof} For all $j =  1,\dots,m$, consider $U^g_j>\|g_j\|(1+e\tr B)^{\deg g_j}$, where $B_i:=\max \{|U_i|,|L_i|\}$ for $i = 1,\dots,n$. Then by Corollary~\ref{cor:mainBasicCompact}\ref{cor:mainBasicCompactBox} and Proposition~\ref{prop:HandNneg} there exists $r \ge0$ and $c_{\alpha, \beta, \gamma,\delta} \ge 0$ for $(\alpha, \beta, \gamma,\delta ) \in \N^{2n+2m}_{r}$ such that
\begin{equation*}
 p(x) =
 \sum_{(\alpha, \beta, \gamma,\delta) \in \N^{2n+2m}_{r} } c_{\alpha,\beta, \gamma,\delta} (x-L)^{\alpha}(U - x)^{\beta}g(x)^{\gamma}(U^g-g(x))^\delta.
\end{equation*}
To finish the proof, we notice that $U^g_j-g_j(x)\in \P^+(\Box^n_{L,U})$  by Lemma~\ref{lem:hatUR}\ref{lem:hatUR_b}, for $j = 1,\dots,m$. Hence we can apply Proposition~\ref{prop:HandNneg} to $(U^g-g(x))^\delta$, to express it as a polynomial in $(x-L)$ and $(U-x)$.
\end{proof}

\begin{remark} \label{rem:base_class} When the base class $\K$ consists of polynomials which are globally nonnegative, the expression~\eqref{eq:genHand} is clearly an NNC. However,  expression~\eqref{eq:genHand} provides an NNC even when the polynomials in $\K$ are not globally nonnegative.
Namely, when $\K \subseteq \P(\Box^n_{L,U})$, with $U, L$ given in Proposition~\ref{prop:genHand}.
For instance,  $\K$ can be the polynomials of the form $q(x-L)$, where $q$ is a SAG polynomial.
\end{remark}

The methodology used to construct non-SOS \schmu-type \german{} in Proposition~\ref{prop:genHand} captures the methods of~\cite{Schweighofer2002,Schweighofer2005,ComplexitySchm,ComplexityPut}, which obtain various \germans{} using  algebraic geometry tools to reduce the problem of certifying nonnegativity over a given compact semialgebraic set~$S$ to certifying nonnegativity over the simplex. There, Polya's \german{}~\citep[][Sec. 2.2]{HardLP88} (Theorem~\ref{thm:polya}) is applied to construct NNCs on~$S$.
In fact, thanks to Corollary~\ref{cor:mainBasicCompact}\ref{item:simplexcert}, we could also use the simplex as the simple  set and use Polya's \german{} as in the papers mentioned above.
However, we choose the box since the resulting certificates directly connect to  the recent literature~\cite{chandrasekaran2016relative,dressler2017positivstellensatz}.
One advantage of our approach is that Proposition~\ref{prop:genHand} reduces the work of obtaining
the \germans{} proven in~\cite{Schweighofer2002,Schweighofer2005,chandrasekaran2016relative,dressler2017positivstellensatz}, to checking that the corresponding base class contains the nonnegative constants, which is straightforward in these cases.

\subsubsection{SDSOS \schmu-type \germans{}}
\label{subsubsec:SDSOS}
Several recent papers  have considered  NNCs with DSOS~\citep[][Def. 3.1]{ahmadi2019dsos} and SDSOS~\citep[][Def. 3.2]{ahmadi2019dsos} polynomials used as the base class \cite{ahmadi2019dsos,josz2017counterexample,ahmadi2017response}.
These  subclasses of SOS are convenient for numerical optimization as they can be represented using linear and second-order cone programming instead of SDP (see Figure~\ref{fig:polys}).
However, as shown in~\cite{josz2017counterexample,ahmadi2017response} and Example~\ref{ex:no-naive}, unlike the SOS Putinar's NNC, those  NNCs are not guaranteed to exist for positive polynomials over compact semialgebraic sets, even when the associated  quadratic module is Archimedean~\citep[see, e.g.,][]{scheiderer2009}.
In contrast, setting the base class $\K$ in Proposition~\ref{prop:genHand} to be the DSOS/SDSOS polynomials, we obtain NNCs of the form~\eqref{eq:genHand} that are guaranteed to exist for all positive polynomials over compact sets.

\begin{corollary}[DSOS/SDSOS \schmu-type \german{}]
\label{cor:SDSOS}
Let $g_1, \dots, g_m \in \R[x]$ be such that $S=\{x\in \R^n:  g_1(x)\ge 0,\dots,g_m(x)\ge 0\}$ is a non-empty compact set. Let $L,U \in \R^n$ be such that  $ S \subseteq \{x \in \R^n:  L \le x \le U\}$.  If $p\in \P^+(S)$,  then there exist $r \ge 0$ and SDSOS (or DSOS) polynomials
$q_{\alpha, \beta, \gamma} \in \R_{r}[x]$ for $(\alpha, \beta,\gamma) \in \N^{2n+m}_{r}$ such that
\begin{equation}
p(x) =
 \sum_{(\alpha, \beta,\gamma ) \in \N^{2n+m}_{r} } q_{\alpha, \beta,\gamma}(x) (x-L)^{\alpha} (U-x)^{\beta}g(x)^\gamma. \label{eq:SDSOS_gen}
\end{equation}
\end{corollary}

 DSOS/SDSOS certificates from \cite{ahmadi2019dsos} provide a sequence of increasing lower bounds for various PO problems when used in~\eqref{eq:NNC-OPT}. However, these bounds do not necessarily converge to the optimal PO problem's (objective) value~\cite{josz2017counterexample,ahmadi2017response}. In contrast, Corollary~\ref{cor:SDSOS} provides a hierarchy of DSOS/SDSOS-based lower bounds that are guaranteed to converge
to the optimal PO problem's value.

The computational benefits of those bounds are shown in Section~\ref{sec:SchNumerics}, where we look at a wide range of PO problems, including problems from the MINLPLib Library  (\url{http://www.minlplib.org/instances.html}) and the class of non-convex quadratic problems introduced in~\citep{yang2018quadratic}.
 For all problems, the LMI hierarchies obtained using~\eqref{eq:SDSOS_gen} performed better than the ones proposed in~\cite{lasserre2001global, ahmadi2019dsos}, in terms of the trade-off between the running times and the quality of the numeric approximation of the objective value. Similar positive numerical results are presented by~\citet{kuang2017alternative}.

\subsubsection{Importance of explicit upper and lower bounds in non-SOS \germans{}}
\label{subsubsec:SDSOS_bounds}

\looseness -1
An essential difference between the NNCs~\eqref{eq:genHand}  from Proposition~\ref{prop:genHand} and existing NNCs such as the \schmu{} NNC~\eqref{eq:schmu} is the explicit dependence on the upper ($U$) and lower ($L$) bounds on elements in $S_g$. These redundant upper and lower bounds appear when lifting to a box. Thus, it is natural to wonder whether they are a byproduct of the method or  necessary to construct NNCs for all polynomials positive on $S$. Notice that the DSOS/SDSOS NNCs introduced in~\cite{ahmadi2019dsos}, which do not use such bounds, have already been  shown not to exist for some positive polynomials~\cite{josz2017counterexample,ahmadi2017response}. The results in this section show that using explicit bounds on elements in $S_g$ is crucial for the existence of non-SOS \schmu-type  certificates.  Thus, we provide an answer to an open question underlying~\citep{dressler2017positivstellensatz, chandrasekaran2016relative}; namely, what types of terms are needed to obtain non-SOS Schm\"{u}dgen-type \germans{}.

We show that a general form of SONC \schmu-type NNCs do not exist for a large class of polynomials positive on $S_g$ if redundant lower and upper bounds on its elements are not explicitly included in the NNC terms (see, eq.~\eqref{eq:schm_josz}). Because the SONC polynomials contain the SDSOS and the DSOS polynomials~\citep[see, e.g.,][Lem. 4.1]{kurpisz2019new}, and because the Putinar's NNC is a subclass of the \schmu{} NNC, the same statement is true for DSOS/SDSOS/SONC Putinar/\schmu{}'s NNCs.

\looseness -1
A polynomial $p \in \R[x]$ is a SONC polynomial~\citep[see, e.g.,][]{wang2018nonnegative,dressler2017positivstellensatz, dressler2018optimization, iliman2016amoebas} if $p=\sum_{i=1}^s p_i$ for some $s>0$, where $p_i$ is a \emph{nonnegative circuit} polynomial for $i=1,\dots,s$.
Namely each $p_i(x)=\sum_{\alpha \in A_i}c_{\alpha}x^{2\alpha} - dx^{\beta}$,
where $A_i \subseteq \N^{n}$ is the vertex set of a simplex, $\beta$ lies in the interior of this simplex, and $c_{\alpha} > 0$,  for all $\alpha \in A_i$.
Nonnegative circuit polynomials are characterized in~\citep[][Thm. 3.8]{iliman2016amoebas} based on their {\em circuit number}.
This characterization leads to membership tests for SONC polynomials based on GP~\citep{dressler2019approach, iliman2016lower, papp2019duality}, or on SOCP~\citep{magron2023sonc, wang2020second}.

Let $g_1,\dots,g_m \in \R[x]$, and let $p\in \R[x]$ be positive on $S_g$. Consider  SONC NNCs for $p$ on $S_g$ of the form
\begin{align}
\|x\|^{2k}p(x)= \sum_{\beta \in \N^{m}_{r} } \sigma_\beta(x) g(x)^\beta, \label{eq:schm_josz}
\end{align}
where $k,r \ge 0$ and for each $\beta \in \N^{n}_{r}$, $\sigma_\beta$ is a  SONC polynomial.
Notice that in~\eqref{eq:schm_josz} we allow for the denominator $\|x\|^{2k}$, but do not explicitly use upper and/or lower bounds for $x \in S_g$.

Next, we prove that certificate~\eqref{eq:schm_josz} does not exist for a large class of $g's$ and $p's$.
In particular, Example~\ref{ex:no-naive} shows that when $S_g$ is the $n$-dimensional ball of radius $R$, the NNC~\eqref{eq:schm_josz} fails even in the case
when $p$ is a quadratic positive polynomial. Our approach is to apply a linear operator from $\R[x]$ to $\R$ to both sides of~\eqref{eq:schm_josz}.
The operator is constructed such that the right-hand side of~\eqref{eq:schm_josz} always gives a nonnegative value, but this is not true for the left-hand side.
For any $\alpha \in \N^n$ we define
\begin{align}
y_\alpha:=\begin{cases}1, & \text{if } |\alpha|  \text{ is odd } \\ (-1)^{\alpha_1}, & \text{otherwise.}     \end{cases} \label{def:mon_seq}
\end{align}
Let $L_y$ be the unique linear operator $L_y: \R[x] \to \R$ such that $L_y(x^\alpha)= y_\alpha$ for all $\alpha\in \N^n$.
First, we show that if a polynomial $p$ has representation~\eqref{eq:schm_josz}, then $L_y(p)$ must be nonnegative (see Proposition~\ref{prop:posLy}).
Then we construct a family of polynomials $p$ positive on $S$ such that  $L_y(p) < 0$ (see Lemma~\ref{lem:L_neg}).

\begin{proposition}\label{prop:posLy} For $j = 1,\dots,m$, let $g_j(x) = \hat g_j(x_1^2,\dots,x_n^2)$  for some polynomial $\hat g_j(x) \in \R[x]$ with $\hat g_j(1,\dots,1) \ge 0$.
If $p\in \R[x]$ has  representation $\eqref{eq:schm_josz}$, then $L_y(p) \ge 0$.
\end{proposition}
\begin{proof}
The proof is based on the following two claims.
\\
{\bf Claim 1. }%\label{lem:LySONC}
If $\sigma\in \R[x]$ is a SONC polynomial, then $L_y(\sigma) \ge 0$.
\\
{\bf Claim 2. } %\label{lem:Ly}
Let $h,f\in \R[x]$ be $n$-variate polynomials, then
$
L_y \left (f(x_1^2,\dots,x_n^2)h(x)\right  ) =f \left (1,\dots,1  \right ) L_y(h(x)).
$

We prove the proposition using these two claims.  Applying $L_y$ to both sides of $\eqref{eq:schm_josz}$ and using Claims 1 and 2, we obtain
\begin{align*}
n^kL_y(p) & =  L_y\left(  \|x\|^{2k}p(x)\right) = \sum_{\beta \in \N^{n}_{r} } L_y\left(  \sigma_\beta(x) g(x)^\beta \right) \\
& = \sum_{\beta \in \N^{n}_{r} } L_y\left(\sigma_\beta(x) \hat g(x_1^2,\dots,x_n^2)^\beta \right) = \sum_{\beta \in \N^{n}_{r} } L_y\left(\sigma_\beta(x) \right) \hat g(1,\dots,1)^\beta \ge 0.
\end{align*}

To finish the proof we need to prove the two claims. To prove Claim 1, notice that by linearity, it is enough to consider nonnegative circuit polynomials. Let  $\sigma(x) = \sum_{\alpha \in A}c_{\alpha} x^{2\alpha} - dx^{\beta}$  where $c_\alpha \ge 0$.  Definition~\eqref{def:mon_seq} of $y_\alpha$ implies that
\begin{align*}
L_y ( \sigma(x) ) & = L_y \left (\sum_{\alpha}c_{\alpha} x^{2\alpha} - dx^{\beta} \right )  =  \sum_{\alpha}c_{\alpha} y_{2\alpha} - dy_{\beta}
 = \sum_{\alpha}c_{\alpha} - d y_{\beta}  = \sigma(y_{\beta},1,\dots,1)  \ge 0.
\end{align*}

To prove Claim 2, first notice that for any $\alpha,\beta \in \N^n$, we have $L_y \left (  x^{2\alpha + \beta} \right ) =  y_{2\alpha+\beta} = y_\beta = L_y \left (  x^{\beta} \right )$.  Write  $f(x):=\sum_{|\alpha|\le \deg f} f_{\alpha} x^\alpha$ and $h(x):=\sum_{|\beta|\le \deg h} h_{\beta} x^\beta$. We have then
\begin{align*}
L_y \left (f(x_1^2,\dots,x_n^2)h(x)\right  ) \ & = \sum_{|\alpha|\le \deg f} \sum_{|\beta|\le \deg h} f_{\alpha}h_{\beta} L_y \left (  x^{2\alpha + \beta} \right )\\
 & = \sum_{|\alpha|\le \deg f} \sum_{|\beta|\le \deg h} f_{\alpha}h_{\beta} L_y \left (  x^{\beta} \right ) =  \sum_{|\beta|\le \deg f} f_{\beta} L_y \left (  h(x) \right ) = f (1,\dots,1)L_y(h(x)).
\end{align*}
\end{proof}

Based on Proposition~\ref{prop:posLy}, to show that~\eqref{eq:schm_josz} does not exist for a given $p \in \R[x]$, it is enough to show that $L_y(p)<0$. The following lemma is a tool to construct such  $p$.

\begin{lemma}\label{lem:L_neg}
Let $q(t)$ be a univariate polynomial. Define the  $n$-variate polynomial $p(x):= q(e\tr x)$.
 Then $L_y(p(x)) = \tfrac 12 (q(n) - q(-n) + q(n-2) + q(2-n))$.
\end{lemma}
\begin{proof}
Let $q(t) = \sum_{k=0}^{d}q_kt^k$. Then $L_y(p(x)) = \sum_{k=0}^{d}q_k L_y((e\tr x)^k)$. However,  note that $L_y((e\tr x)^k) = \allowbreak \sum_{|\alpha|=k}{k \choose \alpha} y_{\alpha}$. Thus, for odd $k$ we have $ L_y((e\tr x)^k) = \sum_{|\alpha|=k}{k \choose \alpha} = n^k$. And for even $k$ we have  $ L_y((e\tr x)^k) = \sum_{|\alpha|=k}{k \choose \alpha}(-1)^{\alpha_1} = (n-2)^k$.
Therefore $L_y(p(x)) = \sum_{k=0}^{d}q_k \tfrac 12(n^k - (-n)^k + (n-2)^k + (2-n)^k) = \tfrac 12 (q(n) - q(-n) + q(n-2) + q(2-n))$.
\end{proof}

Next using Proposition~\ref{prop:posLy} and Lemma~\ref{lem:L_neg} we give several examples of polynomials positive on the ball with no NNC of the form~\eqref{eq:schm_josz}.

\begin{example}[Positive on the ball but no NNC of the form~\eqref{eq:schm_josz}]\label{ex:no-naive} Let $n\ge 2$, and $R \ge n$. Define $g(x):=R - \|x\|^2 $ and $S:=\{x\in \R^n: g(x) \ge 0 \}$. From Proposition~\ref{prop:posLy} and Lemma~\ref{lem:L_neg} the following positive polynomials have not NNC of the form~\eqref{eq:schm_josz}.
\begin{enumerate}[label=(\alph*)]
\item Take $q_1(t): = (t-n)^2+1$. Then $p_1(x) = (e\tr x - n)^2+1$ is an SOS positive everywhere, but it has no representation of the form~\eqref{eq:schm_josz}.
\item Assume $n \le R < k \le n + 2\sqrt {n-1}$. Let $q_2(t) = (k-t)^2 - c$ where $ 0 < c < (k-R)^2$.
Then $q_2(k) = -c <0$, but if $t \le R$, we have $q_2(t) \ge (k-R)^2 - c > 0$. Thus $p_2(x) = q_2(e\tr x)$ is positive on $S$ as for any $x \in S$ we have $e\tr x \le \sqrt{n}\|x\| \le R$. However, $p_2(x)$ has no representation of the form~\eqref{eq:schm_josz}.
\item Let $q_0(t)$ be any univariate polynomial larger than $\tfrac{1}{n^4}$ on $[-R,R]$ . Take $q_3(t) = q_0(t)(t-n)^2(t-n+2)^2(t+n-2)^2+1$. Then $p_3(x) = q_3(e\tr x)$ is positive on $S$. But $p_3(x)$ has no representation of the form~\eqref{eq:schm_josz}.
\end{enumerate}
\end{example}

We conclude this section with an example where the SONC NNC~\eqref{eq:schm_josz} does not exist while the DSOS/SDSOS \schmu-type NNC~\eqref{eq:SDSOS_gen} does.

\begin{example}[SDSOS Schm{\"u}dgen-type  NNC] \label{ex:sdsos}
Let $n=2$ and $0 < c \le 4(\sqrt 2 - 1)^2$. Setting $k=2\sqrt{2}$, $R=2$, and $g(x):=2-x_1^2-x_2^2$, we obtain, from part (b) in Example~\ref{ex:no-naive}, that the polynomial $p_2(x):=(2\sqrt 2-x_1-x_2)^2 - c$  is positive on~$S$, but there is no certificate of the form~\eqref{eq:schm_josz} for the nonnegativity of $p_2(x)$ on~$S$.
In contrast, setting $U=[\sqrt 2,\sqrt 2]\tr$, we construct the following simple SDSOS NNC of the form~\eqref{eq:SDSOS_gen}:
\begin{align*}
p(x) = & 4(\sqrt 2 - 1)^2- c + (6-4\sqrt{2})g(x) +  a g(x)(\sqrt 2-x_1) +a g(x)(\sqrt 2-x_2)\\
& \qquad+ 2\left ( b(1- x_1)^2 +  (a-b)(1- x_2)^2 + \tfrac{1-a}{4}(x_1-x_2)^2 \right ) (\sqrt 2-x_1)\\
& \qquad\qquad+ 2\left ( b(1-x_2)^2 + (a-b)(1-x_1)^2 + \tfrac{1-a}{4}(x_1-x_2)^2 \right ) (\sqrt 2-x_2),
\end{align*}
where $a:=2 - \sqrt{2} > 1$ and $b:=\tfrac{5 -3\sqrt{2}}{4}$.
\end{example}

\subsubsection{Computational experiments}
\label{sec:SchNumerics}

\looseness-1
To benchmark the NNCs obtained from the non-SOS \schmu-type \german{} (Proposition~\ref{prop:genHand}), we perform computational experiments similar to the ones reported in~\cite{kuang2017alternative}. Namely, we use the NNC from equation~\eqref{eq:genHand}  to obtain lower bounds (recall~\eqref{eq:NNC-OPT} and the follow-up discussion) for instances of PO problems considered in~\cite{ghaddar2016dynamic,kuang2017alternative, kleniati2010partitioning, waki2006sums, waki2008algorithm}. We use as base classes SOS, DSOS, and nonnegative constant polynomials; when reporting our results in Table~\ref{tab:compare} they are respectively labeled ``\eqref{eq:genHand}-SOS$_r$'', ``\eqref{eq:genHand}-DSOS$_r$'' and  ``\eqref{eq:genHand}-${\R_+}_r$''.
The total degree (or {\em hierarchy level}) of the NNC~\eqref{eq:genHand} depends on the total degree $r'=e\tr \alpha + e\tr \beta + e\tr \gamma$ of the terms considered in the NNC and the degree of the base class polynomials $c_{\alpha,\beta,\gamma}$. We denote the total degree by $r$. In our implementation, we increase $r$  by increasing $r'$. We set the degree of the base class polynomials $c_{\alpha,\beta,\gamma}$ to be $\min \{2, 2\lfloor 0.5 (r- r') \rfloor \} \le 2$. For comparison purposes, we also use the {\em Lasserre hierarchy}~\cite{lasserre2001global}; that is, Putinar's NNC~\eqref{eq:puti} with total degree $r=2,4,\dots$ (labeled ``\eqref{eq:puti}-SOS$_r$'' in Table~\ref{tab:compare}). Also, we use the NNC with total degree $r=2,4,\dots$, obtained by replacing the SOS polynomials in~\eqref{eq:puti} by DSOS polynomials (labeled ``\eqref{eq:puti}-DSOS$_r$" in Table~\ref{tab:compare}).  The latter approach to computing lower bounds for PO problems was proposed and tested in~\cite{ahmadi2019dsos}. The total degree of the NNC~\eqref{eq:puti} increases with the degree of the polynomials $\sigma_j(x)$, $j=0,1,\dots,m$. In the implementation used here, we set the degree $\sigma_j(x)$ to be equal to $2\lfloor  0.5(r- \deg g_j) \rfloor$, for $j=0,1,\dots,m$ (where $g_0(x) := 1$).
A key characteristic of the considered PO problem instances is that some of these problems are ``challenging'', in the sense that a hierarchy level of $r>2$ is needed for the lower bounds obtained by the Lasserre hierarchy~\cite{lasserre2001global} to converge.
The bounds and times required to solve the LMI problems associated with each hierarchy are reported in Figure~\ref{fig:compare} and Table~\ref{tab:compare}.

\looseness -1
Figure~\ref{fig:compare} shows the trade-off between computation time and tightness of the bound. We observe that for each base class (i.e., SOS or DSOS polynomials) it is substantially better to use the proposed
NNC~\eqref{eq:genHand} than NNC~\eqref{eq:puti}. Sometimes, the Lasserre hierarchy is more effective than the non-SOS \schmu-type hierarchy. However, for larger PO problems, solving higher levels of the Lasserre hierarchy is computationally out of reach, while solving higher levels of the non-SOS \schmu-type hierarchies might be within computational reach, as instead of SDP solvers, SOCP or LP solvers could be used for this purpose.
In Table~\ref{tab:compare}, we can observe a behavior similar to the one illustrated in Figure~\ref{fig:compare} for other PO problems considered in~\cite{ghaddar2016dynamic,kuang2017alternative, kleniati2010partitioning, waki2006sums, waki2008algorithm}.

\begin{figure}[H]
  \begin{subfigure}[b]{0.5\linewidth}
    \centering
      \begin{adjustbox}{width=0.8\textwidth}
    \begin{tikzpicture}
\begin{semilogxaxis}[
  xmin=10^(-2.9),xmax=10^(3.27),
  title style={at={(0.5,1.08)},anchor=north,yshift=-0.1},
  title = Example 1,
  legend cell align=left,
  legend style={nodes={scale=0.7, transform shape}},
  xlabel={Time (s.)},
 % ylabel=Bound,
  grid = both,
%  xtick = {-3, -2,...,2},
%  ytick = {-27, -18,...,0},
 % ymajorgrids, xmajorgrids, yminorgrids, xminorgrids,
  legend pos = south east
  ]
\addplot[mark=*,  dashed, mark options={fill=white, solid}, color = black, line width = \mylinewidth,  mark size=\mymarkersize] table [y=h1b, x=h1t]{Plot_Example1.txt};
\addlegendentry{\footnotesize \eqref{eq:puti}--SOS$_r$ (Lasserre~\cite{lasserre2001global})}
%\addplot[mark=square*,  dashed, mark options={fill=white, solid}, color = black, line width = \mylinewidth,  mark size=\mymarkersize] table [y=h2b, x=h2t]{Plot_Example1.txt};
%\addlegendentry{\footnotesize QM-SDSOS$_r$}
\addplot[mark=triangle*,  dashed, mark options={fill=white, solid}, color = black, line width = \mylinewidth,  mark size=\mymarkersize] table [y=h2b, x=h2t]{Plot_Example1.txt};
\addlegendentry{\footnotesize \eqref{eq:puti}--DSOS$_r$ (Ahmadi et al.~\cite{ahmadi2019dsos})}
\addplot[mark=*, mark options={fill=black, solid}, color = black, line width = \mylinewidth,  mark size=\mymarkersize] table [y=h3b, x=h3t]{Plot_Example1.txt};
\addlegendentry{\footnotesize \eqref{eq:genHand}-SOS$_r$ (Prop.~\ref{prop:genHand})}
%\addplot[mark=square*, mark options={fill=white, solid}, color = black, line width = \mylinewidth,  mark size=\mymarkersize] table [y=h5b, x=h5t]{Plot_Example1.txt};
%\addlegendentry{\footnotesize Po-SDSOS$_r$}
\addplot[mark=triangle*, mark options={fill=black, solid}, color = black, line width = \mylinewidth,  mark size=\mymarkersize] table [y=h4b, x=h4t]{Plot_Example1.txt};
\addlegendentry{\footnotesize \eqref{eq:genHand}--DSOS$_r$ (Prop.~\ref{prop:genHand})}
\addplot[mark=square*, mark options={fill=black, solid}, color = black, line width = \mylinewidth,  mark size=\mymarkersize] table [y=h5b, x=h5t]{Plot_Example1.txt};
\addlegendentry{\footnotesize \eqref{eq:genHand}-${\mathbb{R}_+}_r$ (Prop.~\ref{prop:genHand})}
%\addplot[mark=none, dashed, color = gray, line width = 0.5mm] {-1.57};
 \addplot[mark=none, loosely dashdotted, color = gray, line width = 0.3mm] coordinates {(10^(-2.9), -1.57) (10^(3.27), -1.57)};
  \addlegendentry{\footnotesize Optimal Value}
%  \node[] at (axis cs: 1800 - 400, -4.54) {\tiny $\dag$};
\end{semilogxaxis}

\end{tikzpicture}
\end{adjustbox}

   % \includegraphics[width=0.8\linewidth]{Plot_Example1}
   % \caption{Initial condition}
    \label{fig.compare:a}
  \end{subfigure}%%
  \begin{subfigure}[b]{0.5\linewidth}
    \centering
     \begin{adjustbox}{width=0.8\textwidth}
   \begin{tikzpicture}
\begin{semilogxaxis}[
  xmin=10^(-2.5),xmax=10^(3.27),
  title style={at={(0.5,1.08)},anchor=north,yshift=-0.1},
  title = Example 2,
  legend cell align=left,
  legend style={nodes={scale=0.7, transform shape}},
  xlabel={Time (s.)},
 % ylabel=Bound,
  grid = both,
 % xtick = {-2, -1,...,4},
 % ytick = {-16, -14 ,..., -4},
 % ymajorgrids, xmajorgrids, yminorgrids, xminorgrids,
 % legend pos = south east
  %legend style={at={(0,1)}, anchor=north}
  ]
\addplot[mark=*, dashed, mark options={fill=white, solid}, color = black, line width = \mylinewidth,  mark size=\mymarkersize] table [y=h1b, x=h1t]{Plot_Example2.txt};
%\addlegendentry{\footnotesize QM-SOS$_r$}
%\addplot[mark=square*, dashed, mark options={fill=white, solid}, color = black, line width = \mylinewidth,  mark size=\mymarkersize] table [y=h2b, x=h2t]{Plot_Example2.txt};
%\addlegendentry{\footnotesize QM-SDSOS$_r$}
\addplot[mark=triangle*, dashed, mark options={fill=white, solid}, color = black, line width = \mylinewidth,  mark size=\mymarkersize] table [y=h2b, x=h2t]{Plot_Example2.txt};
%\addlegendentry{\footnotesize QM-DSOS$_r$}
\addplot[mark=*, mark options={fill=black, solid}, color = black, line width = \mylinewidth,  mark size=\mymarkersize] table [y=h3b, x=h3t]{Plot_Example2.txt};
%\addlegendentry{\footnotesize Po-SOS$_r$}
%\addplot[mark=square*, mark options={fill=white, solid}, color = black, line width = \mylinewidth,  mark size=\mymarkersize] table [y=h5b, x=h5t]{Plot_Example2.txt};
%\addlegendentry{\footnotesize Po-SDSOS$_r$}
\addplot[mark=triangle*, mark options={fill=black, solid}, color = black, line width = \mylinewidth,  mark size=\mymarkersize] table [y=h4b, x=h4t]{Plot_Example2.txt};
%\addlegendentry{\footnotesize Po-DSOS$_r$}
\addplot[mark=square*, mark options={fill=black, solid}, color = black, line width = \mylinewidth,  mark size=\mymarkersize] table [y=h5b, x=h5t]{Plot_Example2.txt};
%\addlegendentry{\footnotesize Po-${\mathbb{R}_+}_r$}
 \addplot[mark=none, loosely dashdotted, color = gray, line width = 0.3mm] coordinates {(10^(-2.5), -5.18) (10^(3.27), -5.18)};
% \addlegendentry{\footnotesize Optimal Value}
\end{semilogxaxis}

\end{tikzpicture}
\end{adjustbox}

  %  \includegraphics[width=0.8\linewidth]{Plot_Example2}
  %  \caption{Rupture}
    \label{fig.compare:b}
  \end{subfigure}
  \caption{Legend applies to both plots. Comparison of bound ($y$-axis) vs. time (s.)
  obtained with different hierarchies (of increasing degree) for PO problems Example 1 and 2 in~\cite{ghaddar2016dynamic,kuang2017alternative}. Straight-line progress indicates out-of-time (1800s) or out-of-memory (16GB).}
  \label{fig:compare}
\end{figure}

\begin{table}[H]
{\scriptsize
    \centering
    \begin{tabular}{cccrrrrrrrrrr}
        \toprule
%         &  & &\multicolumn{2}{c}{QM-SOS$_r$}&\multicolumn{2}{c}{QM-DSOS$_r$}&\multicolumn{2}{c}{Po-SOS$_r$}&\multicolumn{2}{c}{Po-DSOS$_r$}&\multicolumn{2}{c}{Po-${\R_+}_r$}\\
& & & \multicolumn{2}{c}{Lasserre~\cite{lasserre2001global}} &  \multicolumn{2}{c}{Ahmadi et al.~\cite{ahmadi2019dsos}} & \multicolumn{6}{c}{non-SOS \schmu-type \german{} (Prop.~\ref{prop:genHand})}\\
\cmidrule(lr){4-5} \cmidrule(lr){6-7} \cmidrule(lr){8-13}
&  & &\multicolumn{2}{c}{\eqref{eq:puti}-SOS$_r$}&\multicolumn{2}{c}{\eqref{eq:puti}-DSOS$_r$}&\multicolumn{2}{c}{\eqref{eq:genHand}-SOS$_r$}&\multicolumn{2}{c}{\eqref{eq:genHand}-DSOS$_r$}&\multicolumn{2}{c}{\eqref{eq:genHand}-${\R_+}_r$}\\
\cmidrule(lr){4-5} \cmidrule(lr){6-7} \cmidrule(lr){8-9}  \cmidrule(lr){10-11}  \cmidrule(lr){12-13}
       Case & $(n,d)$ & $r$ & \multicolumn{1}{c}{Bound} & \multicolumn{1}{c}{T} & \multicolumn{1}{c}{Bound} & \multicolumn{1}{c}{T}  & \multicolumn{1}{c}{Bound} & \multicolumn{1}{c}{T} & \multicolumn{1}{c}{Bound} & \multicolumn{1}{c}{T} & \multicolumn{1}{c}{Bound} & \multicolumn{1}{c}{T} \\\midrule
       \verb"Ex. 3" & (15,2) & 2 & -10.00	& 0.006 & -10.00	& 0.007 &  -8.07	& 0.160 & -9.40	 & 0.015 &-9.40	 & 0.067\\
       & & 4 & -8.07	& 20.569 & -9.65&	6.117 & {\bf -7.43}	& 32.327 & -8.04 &	325.805 & -8.63	& 25.148\\
       &. & 6 & \noVal & \outM & \noVal & \outM & & & \noVal & \outM& \noVal & \outM\\
       \midrule
        \verb"ex2_1_1" & (5,2)  &  2 & \infble & 0.002 &  \infble &0.002 & -18.16	&0.009 &  -18.16&	0.007 &-18.16	&0.006\\
        & & 4 & -17.92	& 0.019 & -18.90	& 0.012 & {\bf -17.00}&	0.505 & -17.13	&0.172 & -17.13&	0.071\\
        & & 6 & {\bf -17.00}	& 0.197 & -18.66&	1.753 & & & {\bf-17.00}	& 36.237 & {\bf-17.00}&	2.660\\
        \midrule
        \verb"ex2_1_2" & (6,2)  &  2 &\infble &	0.002 & \infble &	0.002 & {\bf -213.00}&	0.004 & {\bf-213.00}	&0.002 & {\bf-213.00}	& 0.005\\
               & & 4 &  {\bf-213.00}	& 0.027 &  {\bf-213.00}	& 0.025\\
               \midrule
        \verb"ex2_1_3" & (13,2)  &  2 & \infble &	0.004 & \infble &	0.002 & {\bf -15.00} & 	0.032 & {\bf -15.00} & 	0.005 & {\bf -15.00} &	0.005\\
             & & 4 &  {\bf -15.00}	& 3.156 & {\bf -15.00} 	& 5.199 \\ \midrule
        \verb"ex2_1_4" & (6,2)  &  2 &  \infble &	0.002 &  \infble &	0.002 & {\bf-11.00} &	0.006 & {\bf-11.00}	& 0.003 & {\bf-11.00}	& 0.003\\
               & & 4 &  {\bf-11.00} &	0.040 &  {\bf-11.00} &	0.053 & \\ \midrule
        \verb"ex2_1_5" & (10,2)  &  2 & \infble	& 0.004 & \infble	& 0.002 & {\bf-268.01}&	0.033 & {\bf-268.01}	&0.019 & {\bf-268.01}	& 0.016\\
        & & 4 & {\bf -268.01}	& 0.550 & -268.06	& 2.572 &  \\
        & & 6 & & & \noVal & \outT \\
        \midrule
        \verb"ex2_1_6" & (10,2)  &  2 &  \infble &	0.004 & \infble &	0.002 &   -39.83 &	0.025 & -39.83 &	0.008 & -39.83 &	0.008\\
        & & 4 & {\bf-39.00}	 &0.495 & -40.88&	3.046 &{\bf-39.00}&	8.384&{\bf -39.00}	&28.806&{\bf -39.00}	&8.447  \\
        & & 6 & 	 & & \noVal&	\outT &&	&	&&	&  \\
        \midrule
        \verb"ex2_1_7" & (20,2)  &  2 &  \infble &	0.004 &  \infble &	0.002 & -4334.16	& 0.282 & -4335.53	& 0.114 & -4335.53	& 0.058 \\
         & &4 & \noVal & \outM & \noVal & \outM & \noVal & \outM & \noVal & \outM & \noVal & \outM  \\ \midrule
        \verb"ex2_1_10" & (20,2)  &  2 & \infble &	0.009& \infble & 	0.003& {\bf 49318.02} &	0.211& {\bf  49318.02}&	0.073 & {\bf 49318.02}	& 0.601\\
        & & 4 &  \noVal & \outM & \noVal & \outM  \\\midrule
        \verb"ex3_1_3" & (6,2) &  2 &\infble &	0.003 & \infble &	0.006 & {\bf -310.00} &	0.005 & {\bf -310.00} & 	0.003& {\bf -310.00} &	0.002\\
        & & 4 & {\bf -310.00} &	0.045 & \infble &	0.124 & \\
        & & 6 & & &\infble &	57.851 \\
        & & 8 & & & \noVal & \outT\\ \midrule
        \verb"ex3_1_4" & (3,2)  &  2 & -6.00	&0.001 & -6.00	& 0.002 & -5.69	 & 0.007& -5.69	 & 0.005 & -5.69	&0.006\\
                 & & 4 & -5.69&	0.011 & -5.84	&0.004 & -4.06	&0.147 & -4.55	&0.025 & -4.88	& 0.020 \\
        & & 6 & -4.07	& 0.028 & -5.47	 & 0.032 & {\bf -4.00}	&0.533 & -4.07	& 5.479 & -4.26	&0.171\\
        & & 8 &  {\bf -4.00}&	0.058 & -5.11&	0.320 & & & -4.01	&81.939 & -4.03&	3.318\\ \midrule
        \verb"ex4_1_9" & (2,4) &  4 & -7.00&	0.002 & -7.00&	0.004 & -6.67	&0.015 & -6.67	&0.006 & -6.75	& 0.002\\
        & & 6 & -6.67	& 0.018 & -7.00&	0.005 & {\bf -5.51}	&0.198 & -5.63	& 0.023 & -5.63 & 	0.006\\
        & & 8 & {\bf -5.51}&	0.015 & -7.00	&0.011 & & & -5.57	&0.142 & -5.57& 	0.023 \\ \midrule
        \verb"mathopt2" & (2,4)  &  4 & {\bf 0.00}&	0.003 & {\bf 0.00}	& 0.002 & {\bf 0.00}	& 0.003 & {\bf 0.00}	& 0.003 & {\bf 0.00}&	0.002\\
        \bottomrule
    \end{tabular}
    }
\caption{Comparison of bound vs.\ computation time in sec.\ (T) obtained with different hierarchies for PO problems in~\cite{ghaddar2016dynamic,kuang2017alternative, kleniati2010partitioning, waki2006sums, waki2008algorithm} with $n$ variables and maximum degree (objective and constraints) $d$. Problem's {\bf optimal value} in bold face, ``infeas.'': LMI is infeasible, ``\toutT:'' out-of-time (1800s), ``\toutM'': out-of-memory (16GB). In that last two cases ``$\relbar$'' indicates no bound computed.
}
 \label{tab:compare}
\end{table}

It is worth noting that one could also consider using SDSOS and SONC polynomials as the base class. However, a straightforward implementation of the problems needed to test membership in these classes of polynomials leads to LMIs that are more difficult to solve in comparison to the LMIs needed to test membership in the class of SOS polynomials, defeating the purpose of using SDSOS and/or SONC polynomials. On the other hand, as shown in~\citep[see, e.g.,][]{ahmadi2019dsos, kuang2017alternative, magron2023sonc, LPandMore}, this issue can be overcome by appropriate implementations, but this is out of the scope of this article.

\subsection{A semi-sparse \german{} over non-structured sets} \label{subsec:sparse}

Another avenue that has been explored to make \germans{} more amenable to optimization is to derive structured variations of classical \germans{}.
In particular, sparsity~\citep[see, e.g.,][]{waki2006sums,waki2008algorithm,lasserre2006convergent,wang2019tssos, wang2020chordal,sparse20,Weisser2018,magron2022sparse} and symmetry~\citep[see, e.g.,][]{riener2013exploiting, gatermann2004symmetry} can help to reduce the size of optimization problems obtained with degree-restricted versions of NCCs.

Recall that Putinar's \german{}~\citep{Puti93} states that if, beyond $S_g$ being compact, the quadratic module generated by the polynomials~$g \in \R^m[x]$ is Archimedean, then the NNC~\eqref{eq:schmu} with $\|\alpha\|_1 \le 1$ exists. That is, there exist SOS polynomials $\sigma_{j}$, $j=0,1,\dots,m$, such that
\begin{equation}
\label{eq:puti}
p(x) = \sigma_0(x) + \sum_{j=1}^m \sigma_{j}(x)g_j(x).
\end{equation}
\looseness -1
To further evidence the value of our lifting approach in Section~\ref{sec:approach}, we  set the generalized simplex as a simple  set in Theorem~\ref{thm:mainCompGen}. As a result, we can use Corollary~\ref{cor:mainBasicCompact}\ref{item:simplexcert} to obtain a \german{} (Proposition~\ref{prop:sparse}) that has a {\em semi-sparse} structure, even if the underlying semialgebraic set does not posses any sparsity structure beneficial for the existing results exploiting sparsity~\citep{waki2006sums,waki2008algorithm,lasserre2006convergent,wang2019tssos, wang2020chordal,sparse20,Weisser2018,magron2022sparse}. Our \german{} resembles Putinar's \german{}, but it does not require the Archimedean property of the quadratic module generated by the~$g$'s.

As an intermediate step, we first derive a new sparse SOS \german{} for positive polynomials over the standard simplex that is of independent interest.

\begin{lemma}[Semi-sparse SOS \german{} on the standard simplex]\label{lem:sparseDelta}
Let $F\in \P^+_{d}(\smash{\Delta^n_{\mathbf{0},M}})$. Then there exist $n$-variate  SOS polynomials $\sigma_{0}$ and $\sigma_1$, and  univariate SOS polynomials   $\rho_1,\dots,\rho_{n}, \rho_{n+1}$ such that
\begin{equation}
\label{eq:spacertDelta}
%\begin{split}
F(x) = \sigma_0(x)+\sigma_1(x)\left ( M^2-\|x\|^2 - (M-e\tr x)^2 \right)  %\\
+ \rho_{n+1}(M-e\tr x)(M-e\tr x) +  \sum_{i=1}^{n} \rho_i(x_i)x_i.
%\end{split}
\end{equation}
\end{lemma}

For ease of presentation, we delay the proof of Proposition~\ref{lem:sparseDelta} to the end of this section. The sparsity in~\eqref{eq:spacertDelta} stems from the fact that most of the SOS  required to construct the NNC are univariate polynomials.  However, since the expression in~\eqref{eq:spacertDelta} uses two full-dimensional SOS polynomials (i.e., $\sigma_0, \sigma_1$), we refer to this result as a {\em semi-sparse} SOS \german{}.

Now, by combining Lemma~\ref{lem:sparseDelta} and Corollary~\ref{cor:mainBasicCompact} we obtain a {\em semi-sparse} SOS \german{} on compact sets.

\begin{proposition}[Semi-sparse SOS \german{} on compact sets]
\label{prop:sparse}
Let $g_1, \allowbreak \dots, \allowbreak g_m \in \R[x]$ be such that $S=\{x\in \R^n: g_1(x)\ge 0,\dots,g_m(x) \ge 0\}$ is a non-empty compact set.
Let $L \in \R^n$ and $M\in \R$ be such that  $S \subseteq \{x \in \R^n: L \le x,\, e\tr x \le M\}$. Let $\hat M \ge M+\sum_{j=1}^m \|g_j\| (1+M + e\tr(|L|-L))^{\deg g_j}$ and define
\[
\hat{g}(x) := \left (x-L,g(x), \hat M- e\tr x - e\tr g(x) \right ) \in \R[x]^{n+m+1}.
\]
If $p\in \P^+(S)$, then there exist $n$-variate  SOS polynomials $\sigma_{0}$ and $\sigma_1$, and  univariate SOS polynomials $\rho_1,\dots,\rho_{n+m+1}$ such that
\begin{align}
p(x) & = \sigma_0(x)+\sigma_1(x)\left ( \hat M^2-\sum_{i=1}^{n+m+1} \hat{g}_i(x)^2\right)  +\sum_{i=1}^{n+m+1} \rho_i(\hat{g}_i(x))\hat{g}_i(x).\label{eq:spacert}
\end{align}
\end{proposition}

\begin{proof} Let $ \hat L=[L\tr, \mathbf{0}\tr ] \tr$.
From Corollary~\ref{cor:mainBasicCompact}\ref{cor:mainBasicCompactSimplex}, there exists $F \in \P^+(\smash{\Delta^{n+m}_{\hat L,\hat M}})$ such that $p(x) = F(x,g(x))$.
From Lemma~\ref{lem:sparseDelta}, there exist $n$-variate  SOS polynomials~$\sigma_{0}$ and~$\sigma_1$, and  univariate SOS polynomials $\rho_1,\dots,\rho_{n+m}, \rho_{n+m+1}$ such that
%\[
\begin{equation*}
%\begin{split}
F(y)  = \sigma_0(y)+\sigma_1(y)\left (\hat M^2-\|y\|^2 - (\hat M-e\tr y)^2 \right) %\\
+ \rho_{n+m+1}(\hat M-e\tr y)(\hat M-e\tr y) +  \sum_{i=1}^{n+m} \rho_i(y_i)y_i.
%\end{split}
\end{equation*}
After substituting $y$ with $(x-L,g(x))$ above, the statement of the proposition follows.
 \end{proof}

As in Lemma~\ref{lem:sparseDelta}, the \german{} in  Proposition~\ref{prop:sparse} is semi-sparse, as the SOS polynomial multipliers $\rho_i$, $i=1,\dots,n+m+1$ are all univariate.  A univariate SOS of degree $2d$ can be represented using a $(d+1)\times (d+1)$ positive semidefinite (PSD) matrix~\citep{Nest97}, which is much smaller than the one needed to represent a multivariate SOS of the same degree. Still, the sparse structure of the certificate might come at the cost of large degrees. A compelling line for future work is to compute degree bounds for this type of certificates.
 Also,~the right-hand side of~\eqref{eq:spacert} is linear in (the coefficients of) $\sigma_0$, $\sigma_1$ and $\rho_i$, $i=1,\dots,n+m+1$. Thus, for each fixed degree, the truncated version of NNC \eqref{eq:spacert} is LMI-representable.
When some of the polynomials $g_i$, $i =1,\dots,m$, have few variables, we can make the  certificate~\eqref{eq:spacert} stronger without sacrificing much sparsity by using the following remark.

\begin{remark} \label{rem:verysparse}
In expression~\eqref{eq:spacert} (Proposition~\ref{prop:sparse}), for each $i= 1,\dots,n+m+1$ we have that $\rho_i(\hat g_i(x))$  is an SOS polynomial in $x_{\hat{g}_i}$; that is, in the variables of $\hat{g}_i$. If the cardinality of $x_{\hat{g}_i}$ is small, we could use $\tau_i(x_{\hat{g}_i})$ instead of  $\rho_i(\hat{g}_i(x))$,  where $\tau_i(x_{\hat{g}_i}) \in \R[x_{\hat{g}_i}]$ is an SOS,  to obtain an NNC potentially stronger than~\eqref{eq:spacert}. Notice that this property depends only on the number of variables in each $\hat{g}_i$, $i=1,\dots,m$ and does not require any additional assumptions.
\end{remark}

In the next example, we illustrate the application of Proposition~\ref{prop:sparse} to a compact set for which Putinar’s certificate does not exist for some polynomials positive over the set. Notice that as the set is compact, our NNC~\eqref{eq:spacert} exists for every positive polynomial.

\begin{example}[Certifying nonnegativity over a non-Archimedean set] \label{ex:sparse}
Let $c > \tfrac {17}4$. Consider the polynomials $p(x)=c-x_1^2- x_2^2$, $g_1(x) = x_1- \tfrac{1}{2}$, $g_2(x) = x_2 -\tfrac{1}{2}$ and $g_3(x) = 1-x_1x_2$. Let $S = \{x \in \R^2: g_i(x)\ge 0, i =1,2,3\}$. The set~$S$ has been studied in~\citep[][Ex.~3.1.7]{Fan06} and~\citep[][Ex.~4.6]{jacobi2001distinguished} as an example of a compact semialgebraic set whose quadratic module is not Archimedean. Note that $p \in \P^+(S)$ (as $\|x\| \le \tfrac {\sqrt{17}}2$ for any $x \in S$). But, as the quadratic module associated with the set~$S$ is not Archimedean, Putinar's NNC~\eqref{eq:puti} does not exist for~$p$. In this example,
using the semi-sparse SOS \german{} in Proposition~\ref{prop:sparse}, we construct a simple sparse certificate for~$p$ on~$S$.

Let $L = [1/2,1/2]\tr$ and $M = 5/2$, then $S \subseteq \Delta^2_{L,M}$ and $p\in \P^+(S)$. Let
$\hat M = \tfrac {87}4$ and $\hat g = (x_1-\tfrac 12,x_2-\tfrac 12,g_1,\dots,g_5)$, where $g_4(x) = \tfrac {87}4 - 2 x - 2 y + x y$ and  $g_5(x) = -2 + 89 x - 6 x^2 + 89 y - \tfrac {99}2 xy + 4 x^2 y - 6 y^2 + 4 x y^2  -  2 x^2 y^2$.
It follows from  Proposition~\ref{prop:sparse} that an NNC of the form~\eqref{eq:spacert} can be constructed for $p$ over $S$. In fact, notice that
\[
c -x_1^2- x_2^2 =(c - \tfrac {597}{32}) +  \tfrac 12 (x + y - x y - \tfrac {21}2)^2   + 5 g_3(x) +  \tfrac {47}8 g_4(x)   + \tfrac 14 g_5(x)
\]
That is, for every $c \ge 597/32$ a simple, low-degree NNC of the form~\eqref{eq:spacert} exists.
\end{example}

It is important to mention that in order to derive sparse \germans{} not using full-dimensional SOS polynomials at all, one could use an approach similar to the one  in~\citep[][Sec.~4.2]{roebers2021sparse} to obtain {\em fully-sparse} \germansF{} from semi-sparse ones. This is a topic of ongoing research work.

As mentioned earlier, we finish this section with a proof of Lemma~\ref{lem:sparseDelta}. For that purpose, we first need the lemma below.

\begin{lemma} \label{lem:univ} Let $S\subset \R^n$ be non-empty and compact, and  $p\in \R[x]$. Then $p \in \P^+(S\cap \R_+^n)$ if~and~only~if
\begin{equation}
\label{eq:univ}
p(x)=q(x)+\sum_{i=1}^n x_i \rho_i(x_i),
\end{equation}
where $\rho_1,...,\rho_n$ are  univariate SOS polynomials and $q\in \P^+(S)$.
\end{lemma}
\begin{proof} If $S\subset \R^n_+$, then the result is straightforward. Thus, in what follows, we assume that $S\nsubseteq \R^n_+$.
Without loss of generality, we assume there exists $k\le n$ such that   $\{x \in S,\,x_i <0 \}\neq \emptyset$ for all $i\in \{1,\dots,k\}$, and $\{x \in S,\,x_i <0 \}= \emptyset$ for all $i\in \{k+1,\dots,n\}$.
Since $p \in \P^+(S\cap \R_+^n)$,  there exists $\eps > 0$ such that $x\in S$ and $x>-\eps $ implies $p(x) \ge 0$. Also,  let $M > 0$ be such that $x \in S$ implies $x < M$.
Let $p^0_{\min} = \min\{p(x): x \in S\cap \R_+^n\}$, and let $p^i_{\min}=\min\{p(x) : x \in S,\,x_i \le -\eps \}$ for $i\in \{1,\dots,k\}$.
Consider a function $f_i(t) = a_ite^{-b_it}$  for some $a_i>0$, $b_i>0$, and $i \in \N$.
For any $i\in \{1,\dots,n\}$, we have that $f_i(t)$ is positive  for $t>0$, as well as negative and increasing  for $t<0$.
Therefore, for $i\in \{1,\dots,k\}$, we can tailor $a_i$ and $b_i$ so that $\max \{f_i(x_i):x \in S:x_i \le -\eps\} \le -\eps a_i e^{b_i\eps }  < \smash{\frac {p^i_{\min}}{n}}$ and $\max\{f_i(x_i):x\in S\cap \R_+^n\}\le M a_i <\smash{\frac {p^0_{\min}}{n}}$.
  Defining $f(x) = \sum_{i=1}^k f_i(x_i)$ we obtain $p(x)>f(x)$ for all $x\in S$.

Let $i \in \{1,\dots,k\}$. We show that $f_i(t)$ can be approximated as closely as desired by $t \rho_i(t)$, where $\rho_i$ is a  univariate SOS, which implies $p(x) = q(x) + \sum_{i=1}^n x_i \rho_i(x_i)$ where $q(x) > 0$ for all $x \in S$.

For any $l \ge 0$ consider the Taylor approximation of $e^{t}$ with $2l$ terms:
\[T_{l}(t)=\sum_{j=0}^{2l} \tfrac{t^j}{j!}=1+t+\tfrac{t^2}{2!}+\tfrac{t^3}{3!}+...+\tfrac{t^{2l}}{(2l)!}.\]
Since the Taylor series converges uniformly on bounded intervals, by growing $l$, one can approximate $f_i(t)$ to any desired accuracy by $a_itT_{l}(-b_it)$.
Hence it is enough to show that $T_{l}(t)$ is an SOS, or equivalently, given that $T_{l}$ is a univariate polynomial, that $T_{l} (t)\ge 0$ for all $t$~\citep[see, e.g.,][]{Rezn00}.
We prove the nonnegativity of $T_{l}$ by contradiction.
Assume $T_{l}$ is not nonnegative.
Then $T_l$ must have a zero as $T_{l}(0)=1$.
Let $t^*$ be the largest zero of $T_l$.
Then $t^* < 0$ and $T'_l(t^*) > 0$.
But for any $t$,
$T_{l}'(t)= \smash{\sum_{j=0}^{2l-1} \tfrac{t^j}{j!}} = T_{l}(t) - \smash{\tfrac{t^{2l}}{(2l)!}}$. Thus, $0 < T'_l(t^*) = - \smash{\tfrac{(t^*)^{2l}}{(2l)!}} <0$, which is a contradiction.
\end{proof}

\begin{remark} \label{rem:2}
From the proof of Lemma~\ref{lem:univ} it follows that one could use the same SOS polynomial $\rho_i=\rho$, for $i=1,\dots,n$ in~\eqref{eq:univ}.
\end{remark}

Now, we use the representation from Lemma~\ref{lem:univ} and some of the known NNCs on compact sets from~\cite{PenaVZ08,schmudgen1991k} to prove Lemma~\ref{lem:sparseDelta}.

\begin{proof}[Proof of Lemma~\ref{lem:sparseDelta}] Given $F \in \P_d^+(\Delta^n_{\mathbf{0},M})$ define
$
G(x_0,x) = F\left( \frac {Mx_1}{x_0+e\tr x},\dots,  \frac {Mx_n}{x_0+e\tr x} \right)  (x_0+e\tr x)^{\deg(F)}.
$
 Then $G \in \P_d^+(S \cap V)$, where   $V = \{(x_0,x) \in \R^{n+1}_+: x_0 + e\tr x=M\} = \{(x_0,x)\in \R^{n+1}: (x_0+e\tr x-M)^2=0 \}$, $S = \R_+^{n+1} \cap \mathcal{B}_M$,  and
$\mathcal{B}_M=\{(x_0,x)\in \R^{n+1}:  x_0^2+ \|x\|^2 \le M^2\}$ is the ball of radius $M$ in $\R^{n+1}$.

 It follows from Proposition~\ref{prop:CompEqu} that
\begin{equation}\label{eqloc:sp1}
G(x_0,x)=q(x_0,x)+h(x_0,x)(x_0+e\tr x-M)^2,
\end{equation}
where $h\in \R_{d}[x_0,x]$ and $q \in \P^+_{d+ 2}(\R^{n+1}_+ \cap \mathcal{B}_M)$. Now, by Lemma~\ref{lem:univ},
\begin{equation}\label{eqloc:sp2}
q(x_0,x)=g(x_0,x)+\sum_{i=0}^{n} x_i \rho_i(x_i),
\end{equation}
where $g(x_0,x)\in \P^+(\mathcal{B}_M)$ and $\rho_0,...,\rho_{n}$ are univariate SOS polynomials. Finally, by Schm{\"u}dgen's \german{}~\cite{schmudgen1991k}, we  use NNC~\eqref{eq:schmu} to obtain
\begin{equation}\label{eqloc:sp3}
g(x_0,x)=\sigma_0(x_0,x) + \sigma_1(x_0,x)( M^2 - x_0^2 -\|x\|^2 ),
\end{equation}
where $\sigma_0,\sigma_1$ are SOS polynomials. Replacing $x_0$ by $M - e\tr x$, we have that
\begin{align*}
M^{\deg F}F(x) & = G(M-e\tr x,x)=  q(M-e\tr x,x) & \text{(from~\eqref{eqloc:sp1})} \\
& =  g(M-e\tr x,x) + (M-e\tr x)\rho_0(M-e\tr x) +  \sum_{i=1}^{n} x_i \rho_i(x_i) & \text{(from~\eqref{eqloc:sp2})}\\
& =  \sigma_0(M-e\tr x,x)  + \sigma_1(M-e\tr x,x)( M^2 - (M-e\tr x)^2 -\|x\|^2 ) \\
& \qquad +  (M-e\tr x)\rho_0(M-e\tr x) +  \sum_{i=1}^{n} x_i \rho_i(x_i).   & \text{(from~\eqref{eqloc:sp3})}
\end{align*}
After letting $\rho_{n+1}(t):=\rho_0(t)$, the statement of Lemma~\ref{lem:sparseDelta} follows as  $\sigma_0(M-e\tr x,x)$ and $\sigma_1(M-e\tr x,x)$ are SOS polynomials.
\end{proof}

\subsubsection{Computational experiments} \label{sec:sparseNumerics}
In this section, we compute lower bounds for PO problems using the hierarchy obtained from the semi-sparse SOS NNC~\eqref{eq:spacert}, and the Lasserre hierarchy (i.e., the one obtained~\citep[see, e.g.,][]{lasserre2001global} from Putinar's NNC~\eqref{eq:puti}). We also consider the semi-sparse hierarchy obtained by  using the box $\Box^n_{L,U}$ as the simple set in Proposition~\ref{prop:sparse}. In this case, the resulting NNC will have  the form of~\eqref{eq:spacert} where $ \hat{g}(x) := \left (x-L,U - x,g(x), \hat{U} - g(x) \right ) \in \R[x]^{2n+2m},$ with $\hat U$ as defined in Lemma~\ref{lem:hatUR}\ref{lem:hatUR_b}.
We refer to the two cases of the NNC~\eqref{eq:spacert} as \eqref{eq:spacert}-Simplex and \eqref{eq:spacert}-Box. We call the corresponding hierarchies as \emph{Semi-sparse Simplex} and \emph{Semi-sparse Box}, respectively.

In the semi-sparse hierarchies (resp. Lasserre hierarchy), to obtain the $r$-level of the hierarchy, we choose the degrees of the SOS coefficients in~\eqref{eq:spacert} (resp.~\eqref{eq:puti}) such that the degree of each term in the expression is at most $r$. To test these three approaches, we consider PO problems of the form:
\begin{equation}
\label{def:Sexp}
\min_{x\in \R^n_+} \left \{f(x): a_i x_{2i-1}+b_ix_{2i} +c_i \ge 0, i =1,\dots,\frac{n}{2}, 1-\sum_{i=1}^{\frac{n}{2}}x_{2i}^4 \ge 0, 1-\sum_{i=1}^{\frac{n}{2}}x_{2i-1} \ge 0 \right \},
\end{equation}
where $n \in \{12,16\}$ and $a_i$, $b_i$ and $c_i$, $i=1,\dots, \frac{n}{2}$,  are randomly generated scalars. For each $n\in \{12,16\}$, we generate five different instances of~\eqref{def:Sexp}. For each instance, we minimize four randomly generated objective functions $f(x)$ of degrees one through four. Each objective depends on all variables in the problem, and the coefficients of all random polynomials are generated using independent standard univariate normal distributions.

Instances \eqref{def:Sexp} exhibit some sparsity, but  the structure of this sparsity is not obvious to exploit. Moreover, the objective functions include all possible monomials for the given degree and number of variables.
Thus one cannot efficiently use the existing sparse certificates
from~\cite{waki2006sums,lasserre2006convergent,wang2019tssos,wang2020chordal,sparse20}.
On the other hand, setting $L = 0, U=e$ and $M = n^{3/4}$, the NNCs \eqref{eq:spacert}-Simplex/Box are suitable for such instances.
Notice that for any $g$ and univariate SOS $\rho$,  the polynomial $\gamma(x) = \rho(g(x))$ is an SOS in the same variables as $g$. Thus, in our numerical experiments, when $\hat g_j$ depends on at most two variables, we replace the coefficient $\rho_j(\hat g_j(x))$ in~\eqref{eq:spacert} by an SOS coefficient $\gamma_j(x)$ which depends only on the variables appearing in $\hat g_j$.

In Table~\ref{tabsp4:value}, we compare the lower bounds and computing times obtained using the $r=6$ level of all hierarchies for instances of~\eqref{def:Sexp}. The instances are labeled by {\tt $n\_m$}; that is, by their number of variables $n \in \{14, 15\}$ and the seed $m \in \{4,\dots,8\}$ used to generate their random coefficients (column ``Case'' in Table~\ref{tabsp4:value}). Each instance is set with randomly generated objectives of degree one up to four (columns ``Linear objective'', ``Quadratic objective'', ``Cubic objective'', and ``Quartic objective'' in Table~~\ref{tabsp4:value}).
In all these  hierarchies we reduce the degree of the SOS polynomial $\sigma_0(x)$ to be equal to four, instead of six. Otherwise the Lasserre hierarchy runs out of memory for the largest instances. The degree reduction follows the successful approach used in~\cite{Lass2017} to reduce the size of PO approximations hierarchies.

The effect of using the semi-sparse SOS NNCs~\eqref{eq:spacert}-Simplex/Box in the size of the LMIs required to construct the desired hierarchy versus using the Lasserre hierarchy can be observed in Table~\ref{tabsp4:size}, which lists the number and type of constraints needed to formulate the corresponding LMIs. In particular, the substantial reduction in the size of the LMIs is evident when looking at the number of the largest SDP matrices needed to construct them (column ``Largest SDPs'' in Table~\ref{tabsp4:size}).

As a result of the size reductions reported in Table~\ref{tabsp4:size}, we can observe in Table~\ref{tabsp4:value} that the semi-sparse hierarchies can always be solved in less time than the Lasserre hierarchy (within a 1500s time limit and a 16G memory limit).  The solution time speed-up provided by the semi-sparse hierarchies is obtained with almost no effect on the quality of the lower bound (in comparison with the Lasserre hierarchy). This is highlighted in Table~\ref{tabsp4:value} by placing an $^*$ in the lower bounds provided by the Lasserre hierarchy in the few cases in which the improvement on the bound is more than 1\% relative to the one provided by the strongest semi-sparse hierarchy.

\begin{table}[H]
\centering
{\scriptsize
\begin{tabular}{ccccccccccc}
\toprule
           &              &             \multicolumn{2}{c}{Largest}        & \multicolumn{2}{c}{2nd Largest}                              &   \multicolumn{2}{c}{Other}    &  No. of & No. of\\
&   &\multicolumn{2}{c}{SDPs} &\multicolumn{2}{c}{SDPs}&\multicolumn{2}{c}{SDPs}& non-negativity & equality \\
\cmidrule(lr){3-4} \cmidrule(lr){5-6} \cmidrule(lr){7-8}
Case  & Certificate    & (No.) & Size & (No.) & Size & (No.) & Size & constraints & constraints \\
\midrule
\multirow{3}{*}{\tt 12\_}   & Semi-sparse Simplex      &(1) &  91x91    & (6) &  6x6  &(13) & 3x3 & 2                                                                             & \multirow{3}{*}{18564}                                                               \\
& Semi-sparse Box     &(1) &  91x91    & (12) &  6x6  &(26) & 3x3 & 2                                                                             &                                                             \\ &  Lasserre & (20) &  91x91 & (1) & 13x13   & $(0)$                        & $-$                                                                             & 0                                                            \\ \midrule
\multirow{3}{*}{\tt 16\_}  & Semi-sparse Simplex   & (1) & 153x153  & (8) & 6x6     & (17) & 3x3  & 2 & \multirow{3}{*}{74613}                        \\
& Semi-sparse Box   & (1) & 153x153  & (16) & 6x6     & (34) & 3x3  & 2 &                        \\
  &  Lasserre & (27) & 153x153 & (1) & 17x17    & (0)   & $-$      & 0                       \\ \bottomrule
\end{tabular}
}
\caption{Number and type of constraints needed  to compute lower bounds on instances of~\eqref{def:Sexp} with $n \in \{12,16\}$, using the 6th level semi-sparse hierarchies (i.e., based on NNC~\eqref{eq:spacert}-Simplex/Box) and the Lasserre hierarchy (i.e., based on NNC~\eqref{eq:puti}), with $\deg \sigma_0 = 4$ in the corresponding NNCs.} \label{tabsp4:size}
\end{table}

\begin{table}[H]
\centering
{\scriptsize
\setlength\tabcolsep{2.75pt}
\begin{tabular}{ccccccccccccccccccc}
\toprule
  Case                                                                                                          & \multicolumn{3}{c}{Linear objective} & \multicolumn{3}{c}{Quadratic objective} & \multicolumn{3}{c}{Cubic objective} & \multicolumn{3}{c}{Quartic objective} \\
                       \cmidrule(lr){2-4} \cmidrule(lr){5-7} \cmidrule(lr){8-10} \cmidrule(lr){11-13}
                                           &  \multicolumn{2}{c}{Semi-sparse}          &   {Lasserre}    &  \multicolumn{2}{c}{Semi-sparse}          &   {Lasserre}   &  \multicolumn{2}{c}{Semi-sparse}          &   {Lasserre}   &  \multicolumn{2}{c}{Semi-sparse}          &   {Lasserre}          \\                                                                   \cmidrule(lr){2-3} \cmidrule(lr){5-6} \cmidrule(lr){8-9} \cmidrule(lr){11-12}
                    &  Simplex       &  Box         &    &  Simplex       &  Box         &   &  Simplex       &  Box         &  &  Simplex       &  Box         &          \\ \midrule
{\tt 12\_4}
	&	\begin{tabular}[c]{@{}c@{}} -3.23  \\ (2.9s) \end{tabular}	&	\begin{tabular}[c]{@{}c@{}} -3.23  \\ (2.8s) \end{tabular}	&  \begin{tabular}[c]{@{}c@{}} -3.23  \\ (125.3s) \end{tabular}	&	\begin{tabular}[c]{@{}c@{}} -6.17  \\ (3.0s) \end{tabular}		&	 \begin{tabular}[c]{@{}c@{}} -6.17  \\ (2.8s) \end{tabular}	&\begin{tabular}[c]{@{}c@{}} -6.17  \\ (94.9s) \end{tabular}		&		\begin{tabular}[c]{@{}c@{}}-10.05
	\\ (3.1s) \end{tabular} & \begin{tabular}[c]{@{}c@{}} -8.82	\\ (2.7s)\end{tabular}	& \begin{tabular}[c]{@{}c@{}} -8.82	\\ (94.5s)\end{tabular}		&	\begin{tabular}[c]{@{}c@{}} -16.60	 \\ (2.8s) \end{tabular} &\begin{tabular}[c]{@{}c@{}} -15.32	 \\ (2.8s) \end{tabular}	&		\begin{tabular}[c]{@{}c@{}} -15.28		 \\ (86.6s) \end{tabular}\\ \midrule
{\tt 12\_5} 	
	&	\begin{tabular}[c]{@{}c@{}} -7.78 \\ (2.8s) \end{tabular}	&	\begin{tabular}[c]{@{}c@{}} -7.78 \\ (3.1s) \end{tabular}	&  \begin{tabular}[c]{@{}c@{}} -7.78  \\ (86.6s) \end{tabular}	&	\begin{tabular}[c]{@{}c@{}} -5.27  \\ (2.7s) \end{tabular}		&	 \begin{tabular}[c]{@{}c@{}} -5.23  \\ (3.2s) \end{tabular}	&\begin{tabular}[c]{@{}c@{}} -5.23 \\ (153.7s) \end{tabular}		&		\begin{tabular}[c]{@{}c@{}} -13.08
	\\ (3.9s) \end{tabular} & \begin{tabular}[c]{@{}c@{}}-9.92	\\ (2.8s)\end{tabular}	& \begin{tabular}[c]{@{}c@{}}-9.92	\\ (72.4s)\end{tabular}		&	\begin{tabular}[c]{@{}c@{}} -28.37	 \\ (3.4s) \end{tabular} &\begin{tabular}[c]{@{}c@{}} -12.55		 \\ (3.1s) \end{tabular}	&		\begin{tabular}[c]{@{}c@{}} -12.26$^*$	 \\ (101.5s) \end{tabular}\\ \midrule
{\tt 12\_6}
	&	\begin{tabular}[c]{@{}c@{}} -2.09 \\ (3.0s) \end{tabular}	&	\begin{tabular}[c]{@{}c@{}} -2.09 \\ (2.8s) \end{tabular}	&  \begin{tabular}[c]{@{}c@{}} -2.09 \\ (100.9s) \end{tabular}	&	\begin{tabular}[c]{@{}c@{}} -4.09 \\ (2.8s) \end{tabular}		&	 \begin{tabular}[c]{@{}c@{}}-3.97 \\ (2.9s) \end{tabular}	&\begin{tabular}[c]{@{}c@{}} -3.92$^*$ \\ (79.3s) \end{tabular}		&		\begin{tabular}[c]{@{}c@{}} -8.53
	\\ (3.6s) \end{tabular} & \begin{tabular}[c]{@{}c@{}} -6.14	\\ (3.0s)\end{tabular}	& \begin{tabular}[c]{@{}c@{}} -5.91$^*$	\\ (185.0s)\end{tabular}		
	&	\begin{tabular}[c]{@{}c@{}} -55.64	 \\ (4.1s) \end{tabular} &\begin{tabular}[c]{@{}c@{}} -16.50		 \\ (2.8s) \end{tabular}	&		\begin{tabular}[c]{@{}c@{}} -16.50	 \\ (87.0s) \end{tabular}\\ \midrule
{\tt 12\_7}	
	&	\begin{tabular}[c]{@{}c@{}} -0.32 \\ (2.6s) \end{tabular}	&	\begin{tabular}[c]{@{}c@{}} -0.32 \\ (3.0s) \end{tabular}	&  \begin{tabular}[c]{@{}c@{}} -0.32 \\ (117.1s) \end{tabular}	
	&	\begin{tabular}[c]{@{}c@{}} -7.14 \\ (3.7s) \end{tabular}		&	 \begin{tabular}[c]{@{}c@{}} -6.92 \\ (3.1s) \end{tabular}	&\begin{tabular}[c]{@{}c@{}} -6.87 \\ (93.7s) \end{tabular}		
	&		\begin{tabular}[c]{@{}c@{}} -14.61	\\ (3.8s) \end{tabular} &\begin{tabular}[c]{@{}c@{}} -12.59	 \\ (2.7s) \end{tabular}	&		\begin{tabular}[c]{@{}c@{}} -12.59	 \\ (79.1s) \end{tabular}	
	&	\begin{tabular}[c]{@{}c@{}} -37.48	 \\ (3.4s) \end{tabular} &\begin{tabular}[c]{@{}c@{}} -11.67	 \\ (3.2s) \end{tabular}	&		\begin{tabular}[c]{@{}c@{}} -11.67	 \\ (87.0s) \end{tabular}\\ \midrule
{\tt 12\_8} 	
	&	\begin{tabular}[c]{@{}c@{}} -3.92 \\ (2.5s) \end{tabular}	&	\begin{tabular}[c]{@{}c@{}} -3.92\\ (2.9s) \end{tabular}	&  \begin{tabular}[c]{@{}c@{}}-3.92 \\ (94.6s) \end{tabular}	
	&	\begin{tabular}[c]{@{}c@{}} -10.28 \\ (3.4s) \end{tabular}		&	 \begin{tabular}[c]{@{}c@{}} -10.25 \\ (2.9s) \end{tabular}	&\begin{tabular}[c]{@{}c@{}} -10.25\\ (79.3s) \end{tabular}		
	&		\begin{tabular}[c]{@{}c@{}} -13.94	\\ (3.5s) \end{tabular} &\begin{tabular}[c]{@{}c@{}} -12.00	 \\ (3.1s) \end{tabular}	&		\begin{tabular}[c]{@{}c@{}} -11.85$^*$	 \\ (132.1s) \end{tabular}	
	&	\begin{tabular}[c]{@{}c@{}} -56.83	 \\ (3.7s) \end{tabular} &\begin{tabular}[c]{@{}c@{}}-16.26	 \\ (2.6s) \end{tabular}	&		\begin{tabular}[c]{@{}c@{}} -16.26	 \\ (79.9s) \end{tabular}\\ \midrule
{\tt 16\_4}	
	&	\begin{tabular}[c]{@{}c@{}} -2.19 \\ (26.2s) \end{tabular}	&	\begin{tabular}[c]{@{}c@{}} -2.19 \\ (25.9s) \end{tabular}	&  \begin{tabular}[c]{@{}c@{}} $\relbar$ \\ \outT  \end{tabular}	
	&	\begin{tabular}[c]{@{}c@{}} -5.30 \\ (26.2s) \end{tabular}		&	 \begin{tabular}[c]{@{}c@{}} -5.10 \\ (24.7s) \end{tabular}	&\begin{tabular}[c]{@{}c@{}} -5.10 \\ (1335.2s) \end{tabular}		
	&		\begin{tabular}[c]{@{}c@{}} -10.81	\\ (32.5s) \end{tabular} &\begin{tabular}[c]{@{}c@{}}-9.22	 \\ (24.8s) \end{tabular}	&		\begin{tabular}[c]{@{}c@{}} -9.22	 \\ (1228.6s) \end{tabular}	
	&	\begin{tabular}[c]{@{}c@{}}  -25.39	 \\ (36.1s) \end{tabular} &\begin{tabular}[c]{@{}c@{}}-15.75	 \\ (27.0s) \end{tabular}	&		\begin{tabular}[c]{@{}c@{}} $\relbar$ 	 \\   \outT  \end{tabular} \\ \midrule
{\tt 16\_5}
	&	\begin{tabular}[c]{@{}c@{}} -3.76 \\ (36.5s) \end{tabular}	&	\begin{tabular}[c]{@{}c@{}} -3.76 \\ (27.3s) \end{tabular}	&  \begin{tabular}[c]{@{}c@{}} -3.76 \\ (1245.7s) \end{tabular}	
	&	\begin{tabular}[c]{@{}c@{}}-8.22 \\ (49.6s) \end{tabular}		&	 \begin{tabular}[c]{@{}c@{}} -7.60 \\ (29.4s) \end{tabular}	&\begin{tabular}[c]{@{}c@{}} -7.60 \\ (1335.2s) \end{tabular}		
	&		\begin{tabular}[c]{@{}c@{}} -19.50 \\ (57.7s) \end{tabular} &\begin{tabular}[c]{@{}c@{}}-13.46 \\ (26.8s) \end{tabular}	&		\begin{tabular}[c]{@{}c@{}}-13.46		 \\ (1139.1s) \end{tabular}	
	&	\begin{tabular}[c]{@{}c@{}}  -41.27	 \\ (61.0s) \end{tabular} &\begin{tabular}[c]{@{}c@{}}-15.65	 \\ (28.2s) \end{tabular}	&		\begin{tabular}[c]{@{}c@{}} $\relbar$ 	 \\   \outT  \end{tabular} \\ \midrule
{\tt 16\_6}
	&	\begin{tabular}[c]{@{}c@{}} -3.59 \\ (37.2s) \end{tabular}	&	\begin{tabular}[c]{@{}c@{}} -3.59 \\ (25.1s) \end{tabular}	&  \begin{tabular}[c]{@{}c@{}} $\relbar$ 	 \\   \outT  \end{tabular}
	&	\begin{tabular}[c]{@{}c@{}}-10.47  \\ (30.5s) \end{tabular}		&	 \begin{tabular}[c]{@{}c@{}} -10.37	 \\ (26.8s) \end{tabular}	&\begin{tabular}[c]{@{}c@{}} -10.37	 \\ (1432.2s) \end{tabular}		
	&	\begin{tabular}[c]{@{}c@{}} -25.99 \\ (38.0s) \end{tabular} &\begin{tabular}[c]{@{}c@{}} -14.29 \\ (26.3s) \end{tabular}	&		\begin{tabular}[c]{@{}c@{}} -14.30		 \\ (1321.2s) \end{tabular}	
	&	\begin{tabular}[c]{@{}c@{}}  -35.47	 \\ (34.4s) \end{tabular} &\begin{tabular}[c]{@{}c@{}} -13.58 \\ (27.1s) \end{tabular}	&		\begin{tabular}[c]{@{}c@{}} -13.38$^*$		 \\  (1224.3s)  \end{tabular} \\ \midrule
{\tt 16\_7}
	&	\begin{tabular}[c]{@{}c@{}} -2.80 \\ (26.9s) \end{tabular}	&	\begin{tabular}[c]{@{}c@{}} -2.80 \\ (29.4s) \end{tabular}	&  \begin{tabular}[c]{@{}c@{}} $\relbar$ 	 \\   \outT  \end{tabular}
	&	\begin{tabular}[c]{@{}c@{}} -7.08\\ (30.2s) \end{tabular}		&	 \begin{tabular}[c]{@{}c@{}} -6.69	 \\ (27.1s) \end{tabular}	&\begin{tabular}[c]{@{}c@{}} -6.66 \\ (1431.2s) \end{tabular}		
	&	\begin{tabular}[c]{@{}c@{}} -14.45 \\ (31.6s) \end{tabular} &\begin{tabular}[c]{@{}c@{}} -13.79 \\ (26.9s) \end{tabular}	&		\begin{tabular}[c]{@{}c@{}} $\relbar$ 	 \\   \outT  \end{tabular}
	&	\begin{tabular}[c]{@{}c@{}} -47.95	 \\ (33.0s) \end{tabular} &\begin{tabular}[c]{@{}c@{}}-11.60  \\ (25.9s) \end{tabular}	&		\begin{tabular}[c]{@{}c@{}} $\relbar$ 	 \\   \outT  \end{tabular}  \\ \midrule
{\tt 16\_8}
	&	\begin{tabular}[c]{@{}c@{}} -3.44 \\ (26.2s) \end{tabular}	&	\begin{tabular}[c]{@{}c@{}} -3.44 \\ (27.0s) \end{tabular}	&  \begin{tabular}[c]{@{}c@{}} -3.44 	 \\   (1428.4s)   \end{tabular}
	&	\begin{tabular}[c]{@{}c@{}}-5.84 \\ (26.3s) \end{tabular}		&	 \begin{tabular}[c]{@{}c@{}}-5.52	 \\ (25.9s) \end{tabular}	&\begin{tabular}[c]{@{}c@{}} -5.26$^*$ \\ (1211.7s) \end{tabular}		
	&	\begin{tabular}[c]{@{}c@{}} -12.90\\ (32.5s) \end{tabular} &\begin{tabular}[c]{@{}c@{}} -12.49 \\ (28.2s) \end{tabular}	&		\begin{tabular}[c]{@{}c@{}} -12.46	 \\  (1420.3s)  \end{tabular}
	&	\begin{tabular}[c]{@{}c@{}} -27.85	 \\ (36.1s) \end{tabular} &\begin{tabular}[c]{@{}c@{}}-14.43 \\ (26.0s) \end{tabular}	&		\begin{tabular}[c]{@{}c@{}} -14.23$^*$	 \\   (1101.6s) \end{tabular}  \\
\bottomrule
 \end{tabular}
}
\caption{Comparison of bound  and computation time in sec., shown in parentheses, for instances~\eqref{def:Sexp},
with $n \in \{12,16\}$, using the 6th level semi-sparse hierarchies (i.e., based on NNCs~\eqref{eq:spacert}-Simplex/Box) and the Lasserre hierarchy (i.e., based on NNC~\eqref{eq:puti}), with $\deg \sigma_0 = 4$ in the corresponding NNCs.
} \label{tabsp4:value}
\end{table}

\section{The unbounded case}
\label{sec:copocert}

Given the key role that compactness plays for SOS certificates and their alternatives, a question that has attracted much research attention is which NNCs exist over \emph{unbounded}  semialgebraic sets.   In particular, \citet{Powe04, Mars10} derive a \german{} for the case where the underlying domain is
a cylinder with a compact cross-section. \citet{powers12} derive a~\german{} for the case where the underlying domain is a strip or a half-strip.  More general~\germans{} for possibly noncompact sets based on Putinar's~\citep{Puti93}
 and Schm\"udgen's~\citep{schmudgen1991k} \germans{} are presented in~\citep{NieDS06, DemmNP07, Mars09, moreIdeals1, Guo}. The latter results exploit gradient, Jacobian and KKT ideals.
More recently, \germans{} for noncompact semialgebraic sets are shown to exist if a specific modification of the set is compact~\citep{JeyaLL13} or if the polynomial of interest is coercive~\citep{jeya14}. Two other examples of research in this direction closely related to our work are the results in~\citep{putinar1999solving, dickinson2015extension}, which could be interpreted as constructing very particular liftings to higher dimensions. Also, \citep{mai2021positivity} presents new proofs for results in~\citep{putinar1999solving}  that provide degree bounds on the SOS polynomials used in the \germans{}.

In Theorem~\ref{thm:ineqGen} of Section~\ref{subsec:results_nc}, we derive a lifting for unbounded sets similar to the one for compact sets. In Section~\ref{sec:genpolya}, we use this lifting  to generalize P{\'o}lya's \german~\citep[][Sec. 2.2]{HardLP88} to positive polynomials on any generic semialgebraic set. This allows the use of SOS, SONC, SAG, and DSOS/SDSOS polynomials as the base class to  certify the nonnegativity of polynomials over unbounded semialgebraic sets. Our results require assumptions which hold generically (in the sense that it holds for almost all instances) for polynomials that are positive over the semialgebraic set of interest (Corollary~\ref{cor:generic}).

A subtle difference between the compact and the unbounded case, which makes the construction of NNCs for positive polynomials on $S$ difficult is that when $S$ is unbounded  $\P^+_d(S)$ does not coincide with the interior of $\P_d(S)$. Indeed, in the unbounded case, the constant polynomial $1 \notin \inter \P_d(S)$ for any $d$.
Not surprisingly, one needs more than positivity to construct liftings in the case of unbounded sets. Namely, the polynomial must be {\em strongly positive} (cf. Definition~\ref{def:posInf}). Thus, we begin this section by introducing strong positivity in Section~\ref{sec:strongpos}. Later, in Section~\ref{subsec:results_nc}, we show that this condition is sufficient for the existence of the proposed lifting for unbounded sets (Theorem~\ref{thm:ineqGen}). Section~\ref{subsec:generic} takes a deeper look at  strong positivity. In general, not every polynomial in $\inter \P_d(S)$ is strongly positive. Still, we prove that strong positivity captures the interior generically. Also, we show large classes of unbounded examples where this is true (Proposition~\ref{prop:suffCond}).
Consequently, the proposed certificate is guaranteed to certify the nonnegativity of polynomials in cases from which results by \citet{Netz09} show that certificates of the form~\eqref{eq:puti} do not exist.

\subsection{Strong positivity}
\label{sec:strongpos}

To analyze the sign of a polynomial $p$ on an unbounded set  $S$, we need to consider the \emph{asymptotic behavior} of $p$ on $S$, or in other words, the \emph{behavior of $p$ at infinity}~\citep[see, e.g.,][]{Rezn00}.  For that purpose, we introduce some additional terminology. Given a polynomial $p \in \R[x]$, let $\tilde p$ denote the homogeneous component of $p$ of the highest total degree. That is, $\tilde p$ is obtained by dropping from $p$ all the terms whose total degree is less than $\deg p$. Now, we can analyze the sign of $p$ as follows. Let $y \in \R^n$.
If $\tilde p(y)>0$ then there is $t_0 \in \R$ such that $p(ty)>0$ for all $t>t_0$, since the homogeneous component of the highest degree will eventually dominate the behavior of $p$. Similarly, if $\tilde p(y)<0$, then $p$ will become eventually negative in the $y$-direction. However, if $\tilde p(y)=0$, we do not know how $p(ty)$  behaves when $t \in \R$ tends to infinity. The next definition introduces a concept that allows analyzing the asymptotic behavior of polynomials on semialgebraic sets

\begin{definition}\label{def:tilde} Let $g_1, \dots, g_m, h_1,\dots,h_r \in \R[x]$ and let  $S = \{x \in \R^n: g_1(x)\ge 0,\dots,g_m(x) \ge 0, h_1(x)=0,\dots,h_r(x) =0\}$.
 We denote by $\tilde{S}$ the following set
\begin{align}
\tilde{S}  &=\left \{ x\in \R^n: \tilde g_1(x) \ge 0,\dots,\tilde g_m(x) \ge 0,  \tilde h_1(x)=0,\dots,\tilde h_r(x) =0 \right \}. \label{def:sTilde}
\end{align}
\end{definition}

Now, we are ready to introduce the main condition needed for the existence of liftings over unbounded sets.
\begin{definition}[Strong positivity]\label{def:posInf} A polynomial $p \in \R[x]$ is \emph{strongly positive} on~$S$ if
\begin{equation}
p \in \P^+(S ) \text{ and }  \tilde p \in \P^+(\tilde S \setminus \{0\}). \label{eq:poly}
\end{equation}
\end{definition}

Strong positivity has been used in related results in which the set $S$ might be unbounded (see, e.g., \citep[][Thm~4.2]{putinar1999positive},  \citep[][Prop.~3.5]{Dickinson2018}) as it captures the behaviour of polynomials at  infinity~\citep[see, e.g.,][]{Rezn00}. In particular, strong positivity is a sufficient condition to guarantee the existence of NNCs in~\cite{Dickinson2018}.
As discussed in Section~\ref{subsec:generic}, this condition holds {\em generically}. The precise meaning of generical conditions is given in Definition~\ref{def:genericity}.

\subsection{A lifting for the unbounded case} \label{subsec:results_nc}
Next, we present an extension of Proposition~\ref{prop:CompEqu} to the case of general, not necessarily compact, sets.  In contrast to the compact case, the noncompact case requires additional assumptions. Namely,  the polynomial $p$ has to be strongly positive on $S$ and of a large enough degree. Note that this latter condition is mild as it can be satisfied by multiplying $p$ by a polynomial that is positive everywhere  of an appropriate degree (e.g., $(1 + \sum_{i=1}^n x_i^2)^k$ for some fixed $k \in \N$). Analogously to Section~\ref{sec:compactcase} for the compact case, we build upon an existing result for sets defined with equalities.

\begin{proposition}[{\citep[][Lem. 4.4]{OlgaThesis}}]\label{prop:UbEqu} Let $h_1,\dots,h_m \in \P_d(\R^n_+)$ be such that $S=\{x\in \R^n_+:  h_1(x) = 0,\dots,  h_m(x) = 0\}$ is non-empty. Let  $p\in\R_{=d}[x]$ be strongly positive on $S$. Then there is $F\in\R_{=d}[x]$ strongly positive on $\R^{n}_+$ and $\alpha_j\in \R_{d-\deg g_j}[x]$ for  $j = 1,\dots,m$ such that
\[p(x) = F(x)+\sum_{j=1}^m \alpha_j(x) h_j(x).\]
\end{proposition}

Proposition~\ref{prop:UbEqu} generalizes to unbounded sets Proposition~\ref{prop:CompEqu} used in Section~\ref{sec:compactcase} for compact sets (when $S = \R^n_+$ in Proposition~\ref{prop:CompEqu}). Analogously to  the compact case, Proposition~\ref{prop:UbEqu} is the stepping stone to obtain a lifting.

\begin{theorem}
\label{thm:ineqGen}
Let $g_1,\dots,g_m \in \R[x]$ be such that $S=\{x\in \R^n: g_1(x)\ge 0,\dots,g_m(x) \ge 0\}$ is non-empty. Let  $p\in \P^+(S)$ such that $\tilde p\in \P^+(\tilde{S}  \setminus \{0\})$ be given, and define $d_{\max}:=\max \{\deg g_1,\dots,\deg g_m, \allowbreak \lceil \tfrac{\deg p}{2} \rceil \}$. Then there exists  $F \in \P^+_{2d_{\max}}(\R^{2n+m}_+)$  such that $\tilde F(y,z,u_1^{d_{\max}}, \dots, u_m^{d_{\max}}) \in \P^+_{2d_{\max}}(\R^{2n+m}_+\setminus \{0\})$ and
\begin{multline}
\label{eq:unblift}
 (1+e\tr y+e\tr z)^{2d_{\max}-\deg p} p(y-z) =
F \big  (y,z,(1+e\tr y+e\tr z)^{d_{\max}-\deg g_1} g_1(y-z), \dots,  \\ (1+e\tr y+e\tr z)^{d_{\max}-\deg g_m} g_m(y-z)\big ).
\end{multline}
Moreover, if $S  \subseteq \R^{n}_+$, we can replace $z$ by 0.
\end{theorem}
The proof of Theorem~\ref{thm:ineqGen} is given in Section~\ref{sec:proofsUnbd}.

The lifting constructed in Theorem~\ref{thm:ineqGen} uses rational polynomial expressions since the polynomial $p$ in~\eqref{eq:unblift} has a multiplier $(1+e\tr y+e\tr z)^{2d_{\max}-\deg p}$, also called \emph{a denominator}. As a result, the \germans{} derived using Theorem~\ref{thm:ineqGen} will have a similar denominator too. It is important that this denominator is known and fixed. The existence of  rational NNCs for semialgebraic sets is guaranteed by the Krivine-Stengle \german{}~\citep{KrivPos64,stengle1974nullstellensatz}.
 However, the problem of finding Krivine-Stengle certificates is not tractable in general because the denominator is unknown~\citep[see, e.g.,][for more details]{JeyaLL13}.
In contrast, the  rational \german{} we derive from the lifting in Theorem~\ref{thm:ineqGen} has a known denominator whose degree will be fixed when using degree-restricted versions of the NNC (see, Proposition~\ref{prop:genPolya}).

Theorem~\ref{thm:ineqGen}  requires the polynomial $p$ to be strongly positive over the set $S$. This condition is sufficient for the lifting to exist. The next example illustrates why strong positivity is also necessary.

\begin{example}[Strong positivity as a necessary condition]
\label{ex:simpleex}
Let $S = \{(x_1,x_2) \in \R^2_+: g(x_1,x_2):=(x_1x_2+1)(x_1-x_2)^2 \ge 0, -g(x_1,x_2) \ge 0\}$ and $p(x_1,x_2) = 1+ x_1^3 - (x_2-x_1)^3 \in \R_{=3}[x_1.x_2]$.  We have $\tilde{S} = \{(x_1,x_2) \in \R^2_+: \tilde{g}(x_1,x_2) =0\} = \{(t,t), (0,t), (t,0): t\ge 0\}$. Note that for any $t\ge0$, $p(t,t) = 1+t^3 >0$; that is, $p \in \P_3^+(S)$ (i.e., $p$ is positive on $S$). On the other hand,
for any $t>0$, $\tilde{p}(0,t) = -t^3 < 0$; that is, $\tilde{p} \notin \P_{3}^+(\tilde{S} \setminus (0,0))$ (i.e., $p$ is not strongly positive on $S$). We claim that as a result of the latter fact, there is no $F(x_1,x_2, u_1,u_2) \in \P_d(\R^n_+)$ such that  (cf. Theorem~\ref{thm:ineqGen})
\begin{equation}
\label{eq:certex}
(1+x_1+x_2)^{d-3}p(x_1,x_2) = F(x_1, x_2, g(x_1,x_2), -g(x_1,x_2)).
\end{equation}
 To show this, pick $F \in \R_d[x,u]$ such that~\eqref{eq:certex} holds. Evaluate the expression in~\eqref{eq:certex} at $(x_1, x_2) = (-\frac{1}{t}, t)$, multiply it by $t^d$, and take the limit when $t \to \infty$ to obtain that
\begin{equation}
\label{eq:certex2}
\lim_{t \to \infty} t^d \left (1-\frac{1}{t}+t \right )^{d-3} \left(1 -\frac{1}{t^3} - \left (t+\frac{1}{t} \right )^3\right) = \lim_{t \to \infty} t^{d} F\left (-\frac{1}{t}, t, 0, 0 \right ).
\end{equation}
Note that both expressions inside the limits in~\eqref{eq:certex2} are univariate polynomials in $t$. Therefore, their  behaviour as $t \to \infty$ is given by the homogeneous component of the largest degree of the two polynomials. Namely, from~\eqref{eq:certex2} it follows that
\begin{equation}
\label{eq:certex3}
\lim_{t \to \infty} -t^{2d}  = \lim_{t \to \infty} c \cdot t^{2d},
\end{equation}
where $c$ is the coefficient of the monomial $x_2^{d}$ in $F(x_1,x_2, u_1,u_2)$. From~\eqref{eq:certex3}, it follows that $c < 0$. However, this implies that $\lim_{t \to \infty} F(0,t,0,0) = \lim_{t \to \infty} c \cdot t^{d} < 0$; that is, $F \notin \P_d(\R^n_+)$.

Note that (loosely speaking) the point $(-\frac{1}{t}, t), t > 0$ used to evaluate~\eqref{eq:certex} corresponts, at infinity, with the point $(0,t), t >0$ used to show that $p$ is not strongly positive. Also note that we consider $F$ in~\eqref{eq:certex} to be of any degree greater than or equal to $\deg p$. Thus, this example illustrates that the strong positivity condition in Theorem~\ref{thm:ineqGen} cannot be relaxed, even if we allow $F$ to have any arbitrarily large degree.

Finally, let $S' = \{(x_1,x_2) \in \R^2_+: (x_1-x_2)^2 = 0\}$ and note that since $(x_1x_2 + 1) \ge 0$ for all $(x_1,x_2) \in \R^n_+$, it follows that $S = S'$. However, $\tilde{S'} = \{(t,t): t \in \R_+\} \subset \tilde{S} = \{(t,t), (0,t), (t,0): t\ge 0\}$. Thus, if we use the alternative representation of $S$ given by $S'$, then $p$ is indeed strongly positive on $S'$, and thus liftings exist for $p$ on $S'$.
\end{example}

\subsection{Extending P{\'o}lya's \german{}}
\label{sec:genpolya}

Now, we use the lifting for unbounded sets  in Theorem~\ref{thm:ineqGen}  to derive novel \germans{} for polynomials over possibly unbounded semialgebraic sets. Specifically, we  generalize
P{\'o}lya's \german{}~\citep[][Sec. 2.2]{HardLP88}, which can be used to certify nonnegativity of a homogeneous polynomial over the nonnegative orthant, to an NNC for strongly positive polynomials. For that purpose, we begin by stating P{\'o}lya's \german{}.

\begin{theorem}[{P{\'o}lya's Positivstellensatz~\citep[][Sec. 2.2]{HardLP88}}] \label{thm:polya} Let  $G\in \R[x]$ be a homogeneous polynomial such that $G \in \P^+(\R^n_+ \setminus \{0\})$. Then for some $r>0$ all
coefficients of $(e \tr x)^r G(x)$ are nonnegative.
\end{theorem}

P{\'o}lya's \german{} produces an NNC for homogeneous polynomials positive on $\R^n_+$.  To extend this result we need first a \german{} for polynomials, not necessarily homogeneous, positive on $\R^n_+$. Proposition~\ref{prop:PolyaPols} next is an equivalent version of  P{\'o}lya's \german{} for non-homogeneous polynomials.

\begin{proposition}\label{prop:PolyaPols} Let  $G\in \R[x]$ be a polynomial strongly positive on $\R^n_+$. Then for some $r>0$ all
coefficients of $(1 + e \tr x)^r G(x)$ are nonnegative.
\end{proposition}
\begin{proof}
Let $G\in \R[x]$ of degree $d$,  strongly positive on $\R^n_+$, be given. Let $G_0 \in \R[x_0,x]$ be the homogenous polynomial of degree $d$ defined by $G_0(x_0,x_1,\dots,x_n) = x_0^d G(x_1/x_0,\dots,x_n/x_0)$. Now we prove that $G_0 $ is positive on $\R^{n+1}_+ \setminus \{0\}$. Let $s \in \R^{n+1}_+ \setminus \{0\}$. If $s_0 > 0$ we have
$G_0(s) = s_0^d G(s_1/s_0,\dots,s_n/s_0) > 0 $, as $G \in \P^+_d(\R^{n}_+)$. On the other hand, if $s_0 = 0$, we have $G_0(s) = \tilde G(s_1,\dots,s_n)  > 0 $, as $\tilde G \in \P^+_d(\R^{n}_+)$. By Theorem~\ref{thm:polya} there is $r>0$ such that all coefficients of $(x_0 + e\tr x)^rG_0(x_0,x)$ are nonnegative. By replacing $x_0 = 1$ we obtain that all coefficients of   $(1 + e\tr x)^rG_0(1,x) = (1 + e\tr x)^rG(x)$ are nonnegative.
\end{proof}

Notice that we can not just extend Proposition~\ref{prop:PolyaPols} to obtain a \german{} over unbounded sets, as the lifting $F$ in Theorem~\ref{thm:ineqGen} is not strongly positive, instead it is positive on $\R^{2n+m}_+$ with $\tilde F(y,z,u_1^{d_{\max}}, \dots, u_m^{d_{\max}})$ is positive on $\R^{2n+m}_+ \setminus \{0\}$.  These liftings are not necessarily strongly positive on $\R^{2n+m}$. For instance, let $d_{\max}  =4$ and consider $F(x,y,u) = x^2y^2 + u + 1$. Then $F$ is positive on $\R^3_+$ and $\tilde F(x,y,u^{d_{\max}}) = x^2y^2 + u^4$ is positive on $\R^3_+\setminus\{0\}$. But, $\tilde F(x,y,u) = x^2y^2$ is not positive on $\R^3_+\setminus\{0\}$ as $\tilde F(0,0,1) = 0$.
Therefore, we need an extension of Proposition~\ref{prop:PolyaPols} specialized for this type of lifting in Theorem~\ref{thm:ineqGen}.

For any $d, r, m \in \N$, $\kappa \in \N^m$, $t \in \R$, and $u \in \R^m$, let
\[
D^r_{d,\kappa}(t,u)  = \sum_{\beta \in \N^{m}_{\lfloor \frac{r}{d} \rfloor}} {r \choose d\beta} t^{r - \kappa\tr \beta } u^{\beta}.
\]
 The next result is an extension of Proposition~\ref{prop:PolyaPols}, where the denominator is replaced by $D^r_{d}(1+e\tr x,u)$.
\begin{proposition}\label{prop:genPoly} Let $d \ge 0$, $s \ge 0$ and $F  \in \P^+_s(\R^{n+m}_+)$ be such that $\tilde F(x,u_1^{d},\dots,u_m^{d}) \in \P^+_s(\R^{n+m}_+ \setminus \{0\})$. Then for some  $r>0$ there are $c_{\alpha, \gamma} \ge 0$ for $(\alpha, \gamma) \in \N^{n+m}$ with $|\alpha|+d|\gamma| = r+s$ such that
\[
D^r_{d,ed}(1 + e \tr x,u) F(x,u) = \sum_{\substack{(\alpha, \gamma) \in \N^{n+m}\\ |\alpha|+d|\gamma| \le r+s}} c_{\alpha, \gamma} x^{\alpha} u^{\gamma}
\]
\end{proposition}
The proof of Proposition~\ref{prop:genPoly} is given in Section~\ref{sec:proofsUnbd}. Notice that when $d=1$ we have that $D^r_{1,e}(1 + e\tr x,u) = (1 + e\tr x + e\tr u)^r$
and thus Proposition~\ref{prop:PolyaPols} coincides with Proposition~\ref{prop:genPoly} when $d=1$.

With Proposition~\ref{prop:genPoly} at hand, we can now extend P{\'o}lya's \german{} to certify nonnegativity over general semialgebraic sets with any base class containing the nonnegative constant polynomials (i.e., $\R_+$).

\begin{proposition}[P{\'o}lya's \german{} extension]
\label{prop:genPolya}
Let a base class $\K$ satisfying  $\R_+ \subseteq \K \subset \R[x] $ be given. Let $g_1,\dots,g_m \in \R[x]$ be such that $S=\{x\in \R^n: g_1(x)\ge 0,\dots,g_m(x) \ge 0\}$ is non-empty. Let  $p$ be strongly positive on $S$, and let $d_{\max}:=\max \{\deg g_1,\dots,\deg g_m,\lceil \tfrac{\deg p}{2} \rceil \}$. Then
 there exist $r \ge 0$ and $q_{\alpha, \beta, \gamma} \in \K$ for $(\alpha,\beta,\gamma) \in \N^{2n+m}$ with $|\alpha|+|\beta|+d|\gamma| = r+ 2d_{\max}$ such that
\begin{equation}
\begin{split}
D^r_{d_{\max},(\deg g_1, \dots, \deg g_m)}\left (1+e\tr y + e\tr z,g(y-z)\right )(1+e\tr y+e\tr z )^{2d_{\max}-\deg p}p(y-z)\\
= \sum_{\substack{(\alpha,\beta,\gamma) \in \N^{2n+m}\\ |\alpha|+|\beta|+d_{\max}|\gamma| \le r+ 2d_{\max}}} q_{\alpha,\beta,\gamma}(y-z) y^\alpha z^\beta g(y-z)^\gamma.
\end{split} \label{eq:genPolyaCert}
\end{equation}
Moreover if $S  \subseteq \R^{n}_+$ then we can replace $z$ by 0.
\end{proposition}
\begin{proof} It is enough to prove the statement for the case $\K = \R_+$.
Using Theorem~\ref{thm:ineqGen}, we obtain that there exists $F \in \P^+_{2d_{\max}}(\R^{2n+m}_+)$ such that $\tilde F(y,z,u^{\circ d_{\max}}) \in \P^+_{2d_{\max}}(\R^{n+m}_+\setminus \{0\})$ and
\begin{multline} \label{eq.4_7}
(1+e\tr y+e\tr z)^{2d_{\max}-\deg p}p(y-z) = F(y,z,(1+e\tr y + e\tr z)^{d_{\max}-\deg g_1}g_1(y-z), \dots, \\ (1+e\tr y + e\tr z)^{d_{\max}-\deg g_m}g_m(y-z)).
\end{multline}
Applying Proposition~\ref{prop:genPoly}, we obtain that there is $r>0$ and $c_{\alpha,\beta,\gamma}\ge 0$ for  $(\alpha,\beta,d_{\max}\gamma) \in \N^{2n+m}_{r+2d_{\max}}$  such that
\begin{equation} \label{eq.4_8}
D^r_{d_{\max},d_{\max}e}(1+e\tr y+e\tr z,u)F(y,z,u)
= \sum_{\substack{(\alpha,\beta,\gamma) \in \N^{2n+m}\\ |\alpha|+|\beta|+d|\gamma| \le r+ 2d_{\max}}} c_{\alpha,\beta,\gamma}y^\alpha z^\beta u^\gamma.
\end{equation}
Now, we substitute $u_j \leftarrow (1+e\tr y+e\tr z)^{d_{\max}-\deg g_j}g_j(y-z)$, $j=1,\dots,m$ in equation~\eqref{eq.4_8} and combine with equation~\eqref{eq.4_7} to obtain
\begin{align*}
\hspace{15pt}&\hspace{-15pt}D^r_{d_{\max},(\deg g_1, \dots, \deg g_m)}(1+e\tr y+e\tr z,g(y-z))(1+e\tr y+e\tr z)^{2d_{\max}-\deg p}p(y-z) \nonumber \\
= & D^r_{d_{\max},d_{\max} e}(1+e\tr y+e\tr z,(1+e\tr y+e\tr z)^{d_{\max}-\deg g_1}g_1(y-z), \dots, (1+e\tr y+e\tr z)^{d_{\max}-\deg g_m}g_m(y-z)) \nonumber\\
&\times F(y,z,(1+e\tr y + e\tr z)^{d_{\max}-\deg g_1}g_1(y-z), \dots,(1+e\tr y + e\tr z)^{d_{\max}-\deg g_m}g_m(y-z) )\nonumber \\
= & \sum_{\substack{(\alpha,\beta,\gamma) \in \N^{2n+m}\\ |\alpha|+|\beta|+d_{\max}|\gamma| \le r+ 2d_{\max}}} c_{\alpha,\beta,\gamma}y^\alpha z^\beta (1+e\tr y+e\tr z)^{\sum_{j=1}^m \gamma_j(d_{\max} - \deg g_j)}
g(y-z)^\gamma \label{eq:temp_polya_final} \\
= & \sum_{\substack{(\alpha,\beta,\gamma) \in \N^{2n+m}\\ |\alpha|+|\beta|+d_{\max}|\gamma| \le r+ 2d_{\max}}} q_{\alpha,\beta}y^\alpha z^\beta  g(y-z)^\gamma, \nonumber
\end{align*}
where $q_{\alpha,\beta} \ge 0$ for all $(\alpha,\beta)$.  The first equality follows from a direct substitution in the definition of $D^r_{d_{\max},d_{\max} e}$. The last equality follows after expanding $(1+e\tr y+ e\tr z)^
{\sum_{j=1}^m \gamma_j(d_{\max} - \deg g_j)}$.
\end{proof}

Two characteristics of  the non-SOS \german{} in Proposition~\ref{prop:genPolya} are important for optimization. First, once $r$ is fixed, the degrees of both right- and left-hand sides are fixed and known. Second, there is a lot of flexibility for choosing the base class~$\K$, the only requirement to the base class~$\K$ in Proposition~\ref{prop:genPolya} is that~$\K$ contains the nonnegative constant polynomials (see Remark~\ref{rem:base_class_again}).
Thus, the polynomials in $\K$ do not need to be SOS; instead, other classes of polynomials could be used, such as SONC, DSOS/SDSOS and SAG polynomials. As mentioned in Section~\ref{sec:SchGen}, efficient numerical optimization techniques exist  to certify membership of polynomials in these classes.

\begin{remark} \label{rem:base_class_again} When the base class~$\K$ consists of globally nonnegative polynomials, the expression~\eqref{eq:genPolyaCert} is clearly an NNC. However,  \eqref{eq:genPolyaCert} provides an NNC even when the polynomials in~$\K$ are not globally nonnegative.
Namely, when~$\K \subseteq \P(\R^n_+)$.
For instance,  $\K$ can be the SAG polynomials.
\end{remark}

\subsection{Genericity of strong positivity}
\label{subsec:generic}

In what follows, we show that the strong positivity condition introduced in Definition~\ref{def:posInf}, which is
relevant (see Example~\ref{ex:simpleex}) in our results for unbounded sets, holds generically.

First, to connect our strong positivity assumption (see Definition~\ref{def:posInf}) with some relevant results in the literature,
we introduce the notion of \emph{horizon cone}~(see, e.g., {\citet[][]{RockW98, PenaVZ15}}).

  \begin{definition}[\sloppy Horizon cone]
\label{def:horizon}
The horizon cone $S^{\infty}$  of
a given set  $S \subseteq \R^n$ is defined~as:
\begin{equation*}
\begin{split}
S^{\infty}:=\{ y \in \R^n:  \text{ there exist } x^k \in S, \lambda^k \in \R_+, k=1,2,\dots 
\text{ such that }\lambda^k \downarrow 0 \text{ and }   \lambda^k x^k \rightarrow y\}.
\end{split}
\end{equation*}
\end{definition}

We also recall a result that will be key in showing the genericity of the strong positivity condition.

\begin{lemma}[{\citep[][Lem.~1]{PenaVZ15}}]
\label{lem:tildeInf}
Let $S\subseteq \R^n$. If $p\in \R[x]$ is bounded on $S$ from below, then $\tilde p \in \P_{\deg p}(S^{\infty})$.
\end{lemma}

Now, we formally introduce the meaning of genericity in our context.
\begin{definition}[Genericity]
\label{def:genericity}
A property is said to hold generically (or for a generically chosen element) in a  given set if it holds for all  elements of the set, except  for a subset with Lebesgue measure zero.
\end{definition}

This definition of genericity has been commonly used in PO~\citep[see, e.g.,][]{nie2014optimality, Nie2013, WangRational, lasserre2009convex}. In particular, \citet{nie2014optimality} shows that if the quadratic module associated with the underlying feasible set of a PO problem is Archimedean, then the Lasserre hierarchy~\cite{lasserre2001global}
has finite convergence generically. Also, \citet{lasserre2009convex} uses this definition of genericity to characterize the convex sets that are representable with LMIs.

Next, we define the concept of {\em closedness at infinity} that is crucial in characterizing strongly positive polynomials.

\begin{definition}[\citet{WangRational,Guo,Nie2013}, \citet{PenaVZ15}]  \label{def:closed_at_inf} A semialgebraic set $S$ is  called {\em closed at infinity} if $\tilde{S} = S^\infty$.
\end{definition}

Closedeness at infinity is one of the sufficient conditions for hierarchies of relaxations to PO problems proposed in \cite{WangRational,Guo,Nie2013} to converge to the PO's optimal value~\citep[see, e.g.,][Thm~2.5, condition~(d)]{Nie2013} \footnote{To be more precise, the definition of closedness at infinity in  \cite{WangRational,Guo,Nie2013} looks different from Definition~\ref{def:closed_at_inf},  but the definitions are  analogous to each other.  }. In~\cite{Nie2013,WangRational}, this condition is shown to hold generically.

\begin{proposition}[\citet{Nie2013,WangRational}] \label{prop:genericTilde=inf} If the defining polynomials of the semialgebraic set $S$ are chosen generically, then $S^\infty = \tilde S$; that is, $S$ is closed at infinity.
\end{proposition}

From this result and Lemma~\ref{lem:tildeInf}, the genericity of the strong positivity condition follows.

\begin{corollary} \label{cor:generic} Let $d>0$. If the defining polynomials of the semialgebraic set $S$ are chosen generically, and $p$ is chosen generically among polynomials bounded below on $S$, then $\tilde{p} \in \P_{\deg p}^+(\tilde{S}\setminus \{0\})$.
\end{corollary}

\begin{proof}
From Proposition~\ref{prop:genericTilde=inf}, generically $\tilde{S}  =S^\infty$. Now note that $p \in \R_d[x]$ bounded below on $S$ implies $\tilde p \in \P(S^\infty)$ by Lemma~\ref{lem:tildeInf}. Hence,  $\tilde p \in \P(\tilde S)$.  Since $\mathcal{C}:=\{p: \tilde{p} \in \P_{\deg p}(\tilde{S})\}$ is a convex cone with $\inter(\mathcal{C}) = \{p: \tilde{p} \in \P_{\deg p}^+(\tilde{S}\setminus \{0\})\}$,  generically  $\tilde{p} \in \P^+_{\deg p}(\tilde{S}\setminus \{0\})$.
\end{proof}

Even though Proposition~\ref{prop:genericTilde=inf} and Corollary~\ref{cor:generic} hold generically, one can construct a semialgebraic set~$S$ and a polynomial $p$ such that $\tilde{S}\neq S^\infty$ and $p$ is positive but \emph{not} strongly positive on $S$. To illustrate this, we revise  Example~\ref{ex:simpleex}.

\begin{example}[Example~\ref{ex:simpleex} revisited]
Let $S = \{(x_1,x_2) \in \R^2_+: g(x_1,x_2):=(x_1x_2+1)(x_1-x_2)^2 \ge 0, -g(x_1,x_2) \ge 0\}$ and $p(x_1,x_2) = 1+ x_1^3 - (x_2-x_1)^3 \in \R_{=3}[x_1.x_2]$. From the discussion in Example~\ref{ex:simpleex}, we have that:
\begin{enumerate}[label = (\roman*)]
\item $S^{\infty} \subset \tilde{S}$; that is, in this particular case, the set $S$ is not closed at infinity.
\item $p \in \P^+_3(S)$ and $\tilde{p} \notin \P_{3}^+(\tilde{S} \setminus (0,0))$; that is, in this particular case, $p$ bounded below on $S$ does not imply that $p$ is strongly positive on $S$.
\end{enumerate}
\end{example}

On the other hand, there is a wide range of interesting (non-generic) cases in which the results of Corollary~\ref{cor:generic} hold.  To present these cases, we use a characterization of the interior of the set of polynomials that are nonnegative over a given \emph{unbounded} set.

\begin{proposition} \label{prop:IntDescr} Let $S\subseteq \R^n$ be unbounded. Then
\[\inter \P_d(S) = \{ p \in \R_{=d}[x]:
p\in \P^+(S), \ \tilde p \in \P^+(S^{\infty} \setminus \{0\}) \}.\]
\end{proposition}
\begin{proof} First, let $p\in \inter \P_d(S)$. Since $\R_d[x]$ is a finite-dimensional vector space and $\P_d(S)$ is convex, the interior and the algebraic interior of $\P_d(S)$ coincide~\citep[see, e.g.,][Chapter 17]{Holmes}. Therefore,  $p\in \inter \P_d(S)$ if and only if for any $g\in \R[x]$ there is an $\eps>0$ such that $p+\eps g\in \P_d(S)$.  Hence, as $p\in \inter \P_d(S)$, we have $p\in \R_{=d}[x]$ and $p\in \P^+(S)$.
To show that  $ \tilde p \in \P^+(S^{\infty} \setminus \{0\})$, let $y\in \rec{S}$,  $y\neq 0$. Without loss of generality, $y_1>0$.
From $p\in \inter \P_d(S)$, it follows  that for some $\eps>0$ the polynomial $q(x):=p(x)-\eps x_1^{d} \in \P_d(S)$. From Lemma~\ref{lem:tildeInf}, $\tilde{q}\in \P_d(S^{\infty})$, therefore $\tilde{p}(y) \ge \eps y_1^{d} > 0$.
Thus, $\inter \P_d(S) \subseteq  \{ p \in \R_{=d}[x]:
p\in \P^+(S), \ \tilde p \in \P^+(S^{\infty} \setminus \{0\}) \}$.

To show that $\inter \P_d(S) \supseteq \{ p \in \R_{=d}[x]:
p\in \P^+(S),  \tilde p \in \P^+( S^\infty \setminus \{0\}) \}$, let $p \in \R_{=d}[x]$ such that $p\in \P^+(S)$ and $\tilde p \in \P^+(S^\infty \setminus \{0\}).$ For the sake of contradiction, assume  $p\notin \inter \P_d(S)$. Then there exists $q\in \R_d[x]$ such that for $k=1,2\dots$ there exists $x^{k} \in S$ such that
\begin{align*}
p(x^k)-\tfrac{1}{k} q(x^k) <0.
 \end{align*}
The sequence $x^k,  \ k=1,\dots$ must be unbounded. Otherwise, the sequence $x^k$ is contained in a compact set. Thus, $q$ attains a minimum in this set. This, together with the fact that $p\in \P^+(S)$, would contradict the assumption of $p\notin \inter \P_d(S)$.  Therefore, define $\lambda^k:=\tfrac{1}{\|x^k\|},\ k=1,\dots$ so that $\lim_{k \rightarrow \infty} \lambda^k=0$. The sequence $\lambda^kx^k$, $k=1,\dots$ is bounded and thus has a convergent subsequence with a limit $y \in S' := \{y \in S^\infty : \|y\| = 1\}$.
We have then, that for all $\eps >0$,
\[0 \ge  \lim_{k\rightarrow \infty} (\lambda^k)^d (p(x^k)-\eps q(x^k))=\begin{cases}
\tilde p(y), \text{ if } \deg q<d\\
\tilde p(y)-\eps \tilde q(y), \text{ if } \deg q=d.
\end{cases}\]
But $\tilde p \in \P^+(S')$ and $S'$ is compact. Thus, for some $\eps>0$ small enough we obtain a contradiction.
\end{proof}

From Proposition~\ref{prop:IntDescr}, it follows that if $\tilde{S} = S^\infty$, then $\inter \P_d(S) = \{ p \in R_{=d}: p \text{ is strongly positive } \allowbreak \text{on } S\}$. This fact allows  characterizing some cases in which the strong positivity condition in Theorem~\ref{thm:ineqGen} and Proposition~\ref{prop:genPolya} can be replaced by the condition $p \in \inter \P_d(S)$. Moreover, when $S$ is compact, $\inter \P_d(S)= \P^+_d(S)$, and thus, if $\tilde{S} = S^\infty$, positivity of $p$ on $S$ implies the strong positivity condition.

\begin{proposition}[{\citep[][Prop. 4.8]{OlgaThesis}}]\label{prop:suffCond} Let $g_1,\dots,g_m\in \R[x]$ and  $S=\{x\in \R^n: g_1(x) \ge 0,\dots, \allowbreak g_m (x) \ge 0\}$. If any of the following conditions hold, then $\tilde{S} = S^{\infty}$.
\begin{enumerate}[label = (\roman*)]
\item $g_m(x) = N-\|x\|^2$ for some $N>0$. \label{prop:compAs1}
\item  $g_1,\dots,g_m$ are homogeneous. \label{prop:Hom}
\item \label{prop:Polyh1}  $g_j(x)=q^j_1(x)\dotsm q^j_{k_j}(x)$ for some $k_j>0$ and $q^j_1,\dots, q^j_{k_j} \in \R_1[x]$. Notice that in this case $S$ is a union of polyhedra.
\item $n\ge 2$ and $S=\{x\in \R^n: \big( x_n-\sum_{i=1}^{n-1}x_i^2-b\big)q(x) \ge 0,\,x_n \ge 0\}$, where $b\in \R$ and $q\in \R[x]$ is such that $\tilde{q} \in \P^+(\R^n\setminus \{0\})$.   \label{prop:quadSuff}
\end{enumerate}
\end{proposition}

\subsection{Computational examples}
\label{sec:copocertexp}
To illustrate key features of the P{\'o}lya's \german{} extension (Proposition~\ref{prop:genPolya}), we present  two examples. The first example shows that the P{\'o}lya's \german{} extension can be used to construct (LMI approximation) hierarchies to solve PO problems with unbounded feasible sets, and in particular, PO problems for which  hierarchies based on  Putinar's \german{}~\eqref{eq:puti} fail to give a lower bound for the problem.

\begin{example}[Unbounded generalization of Example~\ref{ex:sparse}]
\label{ex:unbounded_mp}
 Let $n\ge 3$, and $p(x):=  \sum_{i=3}^{n}x_i^3 - x_1^2-x_2^2$, $g_1(x): =  x_1 - \tfrac{1}{2}$, $g_2(x):=  x_2 - \tfrac{1}{2}$, $g_3(x): =   1- x_1x_2 $. Consider the problem
\begin{align}
\label{prob:unb1}
 z^*_n=\inf\{p(x):  x \in S: = \{x\in \R^n: g_1(x) \ge 0 , \ g_2(x) \ge 0, \ g_3(x) \ge 0, \ x_3 \ge 0,\dots,x_n \ge 0 \}\}.
\end{align}
This problem generalizes Example~\ref{ex:sparse} to an unbounded set and multiple variables, and its objective $p$ is strongly positive on $\R^n_+$. As in Example~\ref{ex:sparse},  Putinar's certificate for $p$ over $S$ does not exist. The problem's optimal value $z^*_n=-\tfrac{17}{4}$ for any $n\ge 3$ since the optimal solution is attained at $x^*_3 = \cdots = x^*_n = 0$. We numerically tested that Proposition~\ref{prop:genPolya} allows  to certify the nonnegativity of $p(x)+\tfrac{17}{4}$ on $S$ using $r=0$ and $\K = \R_+$, for $n=3,4,\ldots,20$ (i.e., the $r=0$ level of the hierarchy associated with~\eqref{eq:genPolyaCert} using $\R_+$ as the base class gives the optimal value of~\eqref{prob:unb1}). In fact, for all these values of $n$, the certificate obtained has rational coefficients with small denominators. Below, we explicitly present the certificate obtained for the case $n=3$.
\begin{align*}
 ( 1+ x_1+x_2+x_3)\left ( p(x)+\tfrac{17}{4} \right ) = \ &    \tfrac{11}{2}g_3(x) +3 g_1(x) g_2(x)  + 5 g_3(x)x_3   + 18 g_1^2(x) g_2(x) + 16  g_2^2(x) g_1(x) \\
 &+2g_1^2(x) g_3(x) +2g_2^2(x) g_3(x)  + 2 g_1(x) g_3(x) x_3+ 2 g_2(x) g_3(x) x_3   +2x_3^3\\
 & + 5g_1(x) g_2(x) x_3 + 18g_3(x) x_1x_2 +2 g_1(x) g_2(x)x_1x_3+2 g_1(x) g_2(x)x_2x_3 \\
 &+2g_1^3(x) g_2(x) +18 g_1^2(x) g_2^2(x) + 2g_1(x) g_2(x) x_2^2 + g_1(x) x_3^3 + g_2(x) x_3^3 + x_3^4.
\end{align*}
\end{example}

The second example illustrates two features of the proposed P{\'o}lya's \german{} extension (Proposition~\ref{prop:genPolya}). It provides instances in which the Lasserre hierarchy fails to provide good lower bounds for the PO problems (with unbounded feasible set), while the SOS hierarchy derived from the P{\'o}lya's \german{} extension allows us to compute the optimal value of the PO problems. Also, it shows that
the P{\'o}lya's \german{} extension (Proposition~\ref{prop:genPolya}) can be used to obtain or closely approximate
the optimal value of these PO problems with non-SOS based hierarchies.

\begin{example}[Using SOS and non-SOS base classes]
\label{ex:unbounded_new}
  Let $n\ge 2$, and consider the problem
\begin{equation}
\label{pr:exNie}
 z^*_n = \inf  \left \{ \tfrac{1}{n-1}\sum_{i=1}^{n-1}x_i^2+x_n^2 :\sum_{i=1}^{n-1}x_i^2 - x_n\sum_{i=1}^{n-1}x_i - (n-1) \ge 0, \ x_n^2 - 1 \ge 0, \ x_1,\dots,x_n \ge 0 \right \}.
\end{equation}
This problem is inspired by Example 4.5 in~\cite{DemmNP07}. Proposition~\ref{prop:genPolya} applies here since the objective is always strongly positive. The optimal solution and optimal value of~\eqref{pr:exNie} can be obtained analytically. Namely,
\[
 x^*_n=1, x^*_{n-1} =\tfrac{1+\sqrt{1+4(n-1)}}{2}, x^*_{n-2} = \dots = x^*_2 = x^*_1  = 0 \text{ and } z^*_n = 2+\tfrac{1+\sqrt{1+4(n-1)}}{2(n-1)}.
 \]

 Positive coefficients ensure that the Lasserre hierarchy can only use polynomials of degree two since higher degree terms will not cancel out. The Lasserre hierarchy of degree two, for any $n$,  provides the bound of $z^*_n \ge 1.3820$, which,  as can be seen from Table~\ref{tab:exNie}, is substantially below the value of $z_n^*$. Conversely, Table~\ref{tab:exNie}  shows the optimal or very tight bounds on $z^*_n$ that are obtained using
the hierarchy associated with the certificate~\eqref{eq:genPolyaCert} when the base class is set to be the
 SOS, DSOS, or $\R_+$ polynomials of degree less than or equal to 2. In particular, for each base class, Table~\ref{tab:exNie} provides the lower bound ($\lb_0$) provided by first level ($r=0$) of the hierarchy with each of the base classes. Note that when the base class is the SOS polynomials, the lower bound given by the first level of the hierarchy is equal to the optimal value of~\eqref{pr:exNie} (values are presented with a precision of $10^{-4}$). In the cases when the base class is the DSOS or $\R_+$ polynomials, Table~\ref{tab:exNie} provides the lower bound obtained at the first level of the hierarchy ($\lb_0$), as well as the lower bound ($\lb_r$) obtained at a particular hierarchy level $r$. This $r$ is chosen as the hierarchy level that either satisfies: $\lb_r = z^*_n$ (for the first time), computing $\lb_{r+1}$ leads to numerical errors, or  computing $\lb_{r+1}$ takes more than 1800 seconds. These three cases are respectively noted in the status (sts.) columns of the table with a $*$, $\diamond$, or a \outTime. Columns $T$ give the time in seconds needed by the solver to compute the bounds.

\begin{table}[H]
\centering
{\scriptsize
\renewcommand{\tabcolsep}{3.5pt}
\renewcommand{\arraystretch}{1.2}
\begin{tabular}{cccccrcccrccccrcccrccccr}
\toprule
& & && \multicolumn{20}{c}{Base class used in hierarchy associated with~\eqref{eq:genPolyaCert}}\\
\cmidrule{5-24}
& & & & \multicolumn{2}{c}{SOS} && \multicolumn{8}{c}{DSOS} && \multicolumn{8}{c}{$\R_+$}\\
   \cmidrule{5-6} \cmidrule{9-15} \cmidrule{18-24}
$n$ & $z^*_n$ &&  &$\lb_0$ & \multicolumn{1}{c}{T} && & $\lb_0$ & \multicolumn{1}{c}{T} & & $\lb_r$ & $r$ & \status &  \multicolumn{1}{c}{T} && & $\lb_0$ & \multicolumn{1}{c}{T} & & $\lb_r$ & $r$ & \status &  \multicolumn{1}{c}{T} \\
\midrule
2	& 3.6180	&&	& \textbf{3.6180} & $< $0.01   && & 3.5000 & $< $0.01  & & \textbf{3.6180} & 5 & *     & 0.10  && & 3.0000 & $< $0.01 & & 3.6179       & 8  &  $\diamond$  & $< $0.01 \\
3      & 3.0000      &&   & \textbf{3.0000} & $< $0.01  && & 2.9545 & $< $0.01 & & \textbf{3.0000} & 6 & *     & 8.25  &&  & 2.7500 & $< $0.01 & & \textbf{3.0000} & 7 & *    & 0.49      \\
4       & 2.7676      &&   & \textbf{2.7676} & 0.02      && & 2.7500 & 0.02     & & \textbf{2.7676} & 3 & *     & 5.74   && & 2.6000 & $< $0.01 & & 2.7675      & 4 & $\diamond$    & 0.70     \\
5     & 2.6404      &&      & \textbf{2.6404} & 0.05     && & 2.6346 & 0.01     &  & \textbf{2.6404} & 2 & *     & 10.61  && & 2.5000 & $< $0.01 & & 2.6401       & 4 &  $\diamond$   & 7.34       \\
6     & 2.5583      &&      & \textbf{2.5583} & 0.08     && & 2.5571 & 0.42      & & \textbf{2.5583} & 2 & *     & 366.69 && & 2.4000 & 0.01     & & 2.5578      & 5 & \outTime     & 1279.20   \\
7    & 2.5000     &&       & \textbf{2.5000} & 0.18      && & \textbf{2.5000} & 1.09      &   & \textbf{2.5000} & 0 &*     & 1.09 && & 2.3333 & 0.03     & & 2.4985      & 4  & \outTime  &  1236.31    \\
\bottomrule
\end{tabular}
}
\caption{Lower bounds ($\lb_0$, $\lb_r$), at hierarchy level $r$, for problem~\eqref{pr:exNie} using certificates \eqref{eq:genPolyaCert} from Proposition~\ref{prop:genPolya} for $n \in \{1,2,\dots,7\}$ with SOS, DSOS (of degree up to two) and $\R_+$ polynomials as the base class. Running time (T) is given in seconds. Status (sts.) indicates: (*) optimal value found at that hierarchy level ($r$), ($\diamond$) (resp. (\outTime)) numerical error (resp. out of time ($1800$s)) reported at next hierarchy level ($r+1$). Optimal values are highlighted in bold face. \label{tab:exNie}}
\end{table}

\end{example}

\subsection{Delayed Proofs from Section~\ref{sec:copocert}}\label{sec:proofsUnbd}
For brevity, in what follows,  given $u \in \R^m$ and~$d \in \N$, we let
\[
u^{\circ d} := [u_1^{d},\dots,u_m^{d}] \tr .
\]

\begin{proof}[Proof of Theorem~\ref{thm:ineqGen}] For ease of presentation, we first assume that $S \subseteq \R^n_+$. For $j = 1,\dots,m$, let $d_j = \deg g_j$, and define
$
w_j(x,u): = ((1+e\tr x)^{d_{\max}-d_j}g_j(x)-u_j^{d_{\max}} )^2 \in \R[x,u].
$\\
Let
$U=\{(x,u) \in \R_+^{n+m}:  w_1(x,u)=0,\dots, w_m(x,u)=0\},$
and let $q(x):=(1+e\tr x)^{2d_{\max}-\deg p}p(x)$. We apply Proposition~\ref{prop:UbEqu} to $U$ and $q$. To do this, we first check that the assumptions of the proposition hold.
First, note that $S$ being non-empty implies that $U$ is non-empty. Also, for any $(x,u) \in U$ we have~$x \in S$ and thus $q(x) > 0 $; that is, $ q \in \P^+{}(U)$. Next, we have that for $j=1,\dots,m$,
$
\tilde{w}_j(x,u)=((e\tr x)^{d_{\max}-d_j}\tilde g_j(x)-u_j^{d_{\max}} )^2.
$\\
Let $(x,u) \in \tilde{U}$. If $x=0$, then $u=0$, and if $x\neq 0$, then $(e\tr x)^{d_{\max}-d_j}\tilde g_j(x)=u_j^{d_{\max}} \ge 0$ for $j= 1,\dots,m$. Therefore $x\in \tilde{S}  $, which implies  $\tilde{q}(x) =(e\tr x)^{2d_{\max}-\deg p}\tilde{p}(x) >0$, since $\tilde p \in \P^+(\tilde S \setminus \{0\})$. Hence $\tilde q\in \P^+{}(\tilde U \setminus \{0\})$.

 Proposition~\ref{prop:UbEqu} implies that there exists~$G \in \P^+_{2d_{\max}}(\R^{n+m}_+)$ such that  $\tilde G \in \P^+_{2d_{\max}}(\R^{n+m}_+\allowbreak \setminus \{0\})$  and $\alpha_j \in \R$, $j=1,\dots,m$,
 such that
 \vspace{-.1cm}
 \begin{equation}
 \label{eq:rrep}
q(x) = G(x,u) +  \sum_{j=1}^m\alpha_j w_j(x,u)
=G(x,u) +  \underbrace{\sum_{j=1}^m\alpha_j \left((1+e\tr x)^{d_{\max}-d_j}g_j(x)-u_j^{d_{\max}}\right )^2}_{R(x,u)}.
\end{equation}
\looseness -1
Since representation \eqref{eq:rrep} of $q(x)$ depends on $x$ and $u$, the $u$ variables have to cancel out on the right-hand side of~\eqref{eq:rrep}. Since $\alpha_j \in \R$ and each $w_j(x,u)$ depends on $u_j$ only, for $j=1,\dots,m$, the monomials with $u_1,\dots,u_m$ in the polynomial $R(x,u)$ do not cancel out with each other. Thus, all these monomials have to cancel out with monomials of $G(x,u)$. Moreover, $G(x,u)$ cannot contain any other monomials with $u_1,\dots,u_m$. Therefore, in all monomials in $G(x,u)$ containing $u$ the degrees of~$u_j$ are~$d_{\max}$ or~$2d_{\max}$, for all $j=1,\dots,m$. Now, replacing $u_j = \left (\smash{(1+e\tr x)^{d_{\max}-d_j}g_j(x)}\right )^{\sfrac 1{d_{\max}}}$  for all  $j=1,\dots,m$, we obtain from  \eqref{eq:rrep}  that
$(1+e\tr x)^{2d_{\max}-\deg p}p(x) = F(x,(1+e\tr x)^{d_{\max}-\deg g_1}g_1(x), \dots,  (1+e\tr x)^{d_{\max}-\deg g_m}g_m(x)),$
where
$F(x,u_1,\dots, u_m):=G(x,u^{\sfrac 1{d_{\max}}},\dots, u_m^{\sfrac 1{d_{\max}}})$ is a polynomial.

Finally, $G \in \P^+_{2d_{\max}}(\R^{n+m}_+)$ implies $F \in \P^+_{2d_{\max}}(\R^{n+m}_+)$. Also,
$
\tilde{F}(x,u^{\circ d_{\max}}) =  \tilde G \in \P^+_{2d_{\max}}(\R^{n+m}_+\setminus \{0\}).
$
Note that since $F(x,u^{\circ d_{\max}})=G(x,u)$, then
$F(x,u^{\circ d_{\max}})$ and thus $\tilde F(x,u^{\circ d_{\max}})$ have degree $2d_{\max}$.
 To finish the proof, we relax the assumption $S  \subseteq \R^{n}_+$. Define $T:=\{(y,z)\in \R^{2n}_+: g_1(y-z)\ge 0,\dots,g_m(y-z)\ge 0\}=\{(y,z)\in \R^{2n}_+: y-z \in S\}.$ Then $T$ is non empty, $p(y-z)\in \P^+(T)$ and $\tilde p(y-z)\in \P^+(\tilde T \setminus \{0\})$. The statements follow after noticing that $x \in S$ implies $(\max\{0,x\}, -\min\{0,x\}) \in T$ and $\tilde{T}=\{(y,z)\in \R^{2n}_+: y-z \in \tilde S\}$. Hence we can apply the approach described at the start of the proof to the polynomial $p(y-z) \in \R[y,z]$ and the set  $T \subseteq \R^{2n}_+$.
\end{proof}

\begin{proof}[Proof of Proposition~\ref{prop:genPoly}]
Assume $F  \in \P^+_s(\R^{n+m})$ is such that $\tilde F(x,u^{\circ d}) \in \P^+_s(\R^{n+m}_+\setminus \{0\})$.
Notice that $F(x,u^{\circ d})$ is strongly positive on $\R^{n+m}_+$. Thus, we obtain that $(1 + e \tr x + e\tr u)^{r} F(x,u^{\circ d}) = (1 + e \tr x + e\tr u)^{r} G(1,x,u)$  is a polynomial of degree $r+s$ with nonnegative coefficients. By Proposition~\ref{prop:PolyaPols} there are $c_{\alpha, \gamma} \ge 0$ for $(\alpha,\gamma) \in \N^{n+m}_{r+s}$ such that
\begin{equation}
\label{eq:genPolcs}
(1 + e \tr x + e\tr u)^{r} F(x,u^{\circ d}) =
 \sum_{(\alpha, \gamma) \in \N^{n+m}_{r + s} } c_{\alpha, \gamma} x^{\alpha} u^{\gamma}.
\end{equation}
Now let $\omega = \exp(2\pi i/d)$ be the primitive $d$-root of unity. Consider the summation of all the expressions obtained by  substituting $u_i \leftarrow \omega^{k_i}u_i$ in~\eqref{eq:genPolcs}, where each $k_i = 1,\dots d$. We obtain
\begin{align}
\nonumber\sum_{k_1=1}^{d}\cdots \sum_{k_m=1}^{d} (1 + e \tr x + \omega^{k_1}u_1 + \cdots + \omega^{k_m}u_m )^{r} F(x,(\omega^{k_1}u_1)^d,\dots,(\omega^{k_m}u_m)^d)x\\
\label{eq:genPolcsOmega}
 = \sum_{k_1=1}^{d}\cdots \sum_{k_m=1}^{d}\sum_{(\alpha, \gamma) \in \N^{n+m}_{r+s} } c_{\alpha, \gamma} x^{\alpha} (w_1^{k_1}u_1)^{\gamma_1}\cdots(w_m^{k_m}u_m)^{\gamma_m}.
\end{align}
Now we simplify each of the terms in~\eqref{eq:genPolcsOmega}. First, notice that for each $i$ and $k_i$ we obtain $(\omega^{k_i}u_i)^d = (\omega^d)^{k_i} u_i^d = u_i^d$ and thus
\begin{equation}\label{eq:genPolF}
  F(x,(\omega^{k_1}u_1)^d,\dots,(\omega^{k_m}u_m)^d) = F(x,u^{\circ d}).
\end{equation}
Second, using  that for any $\gamma >0$,
\[\sum_{k=1}^{d} \omega^{k\gamma} = \begin{cases} d  \text{ if $d$ divides }\gamma,\\ 0 \text{ otherwise}, \end{cases}
\]
we obtain
\begin{equation}
\label{eq:genPolDen}
\begin{split}
&\sum_{k_1=1}^{d}\cdots \sum_{k_m=1}^{d}  (1 + e \tr x + \omega^{k_1}u_1 + \cdots + \omega^{k_m}u_m )^{r} \\
&=  \sum_{k_1=1}^{d}\cdots \sum_{k_m=1}^{d} \sum_{\beta \in \N^{m}_{r}} {r \choose \beta} (1 + e \tr x)^{r-e\tr\!\beta}(\omega^{k_1}u_1)^{\beta_1}  \cdots  (\omega^{k_m}u_m )^{\beta_m} \\
& =  \sum_{\beta \in \N^{m}_{r}} {r \choose \beta} (1 + e \tr x)^{r-e\tr\!\beta} \sum_{k_1=1}^{d} (\omega^{k_1}u_1)^{\beta_1}  \cdots \sum_{k_m=1}^{d}   (\omega^{k_m}u_m )^{\beta_m} \\
& =  \sum_{\beta \in \N^{m}_{r}} {r \choose \beta} (1 + e \tr x)^{r-e\tr\!\beta} u_1^{\beta_1}\cdots u_m^{\beta_m} \sum_{k_1=1}^{d} \omega^{k_1\beta_1}  \cdots \sum_{k_m=1}^{d}   \omega^{k_m\beta_m} \\
& = d^m \sum_{\beta \in \N^{m}_{\lfloor \frac{r}{d} \rfloor}} {r \choose d\beta}  (1 + e \tr x)^{r-d e\tr\!\beta}  u_1^{d\beta_1}  \cdots u_m^{d\beta_m}\\
& = d^m D^r_{d,de}(1 + e \tr x,u^{\circ d}).
\end{split}
\end{equation}
Similarly, we have
\begin{equation}
\label{eq:genPolRHS}
\begin{split}
 \sum_{k_1=1}^{d}\cdots \sum_{k_m=1}^{d}\sum_{(\alpha, \gamma) \in \N^{n+m}_{r+s} } c_{\alpha, \gamma} x^{\alpha} (w_1^{k_1}u_1)^{\gamma_1}\cdots(w_m^{k_m}u_m)^{\gamma_m}
 =
d^m\sum_{\substack{(\alpha, \gamma) \in \N^{n+m}\\ |\alpha| + d |\gamma| \le r+s }} c_{\alpha, \gamma} x^{\alpha} u_1^{d\gamma_1}\cdots u_m^{d\gamma_m}.
\end{split}
\end{equation}
Plugging \eqref{eq:genPolF}-\eqref{eq:genPolRHS} into \eqref{eq:genPolcsOmega} and multiplying both sides by $d^{-m}$, we obtain
\begin{equation*}
D^r_{d,ed}(1 + e \tr x,u^{\circ d}) F(x,u^{\circ d}) = \sum_{\substack{(\alpha, \gamma) \in \N^{n+m}\\ |\alpha| + d |\gamma| \le r+s }} c_{\alpha, \gamma} x^{\alpha} u_1^{d\gamma_1}\cdots u_m^{d\gamma_m}.
\end{equation*}
Substituting $u_i \leftarrow u_i^{1/d}$ for $i=1,\dots,m$ we obtain that
\begin{equation*}
D^r_{d,ed}(1 + e \tr x,u) F(x,u) = \sum_{\substack{(\alpha, \gamma) \in \N^{n+m}\\ |\alpha| + d |\gamma| \le r+s }} c_{\alpha, \gamma} x^{\alpha} u_1^{\gamma_1}\cdots u_m^{\gamma_m},
\end{equation*}
a polynomial with nonnegative coefficients.
\end{proof}

\section{Concluding remarks}\label{sec:remarks}
In this paper, we reduce the problem of certifying the nonnegativity of a polynomial over a general semialgebraic set to the problem of certifying the nonnegativity of a related polynomial, called a lifting, over a simpler  set. Using this methodology, novel \germansF{} with advantageous properties are derived.

\looseness-1
We derive a
non-SOS Schm\"{u}dgen-type \germansF{} to certify the nonegativity of polynomials over general compact semialgebraic sets, as well as an analogous result for the case of general unbounded semialgebraic sets; namely, an extension of P{\'o}lya's \germanF{}. These two \germansF{} share one main feature. Unlike related results in the literature, there is a lot of freedom in choosing the
{\em base class} for the \germansF{}; that is, the class of nonnegative polynomials used to construct the associated nonnegativity certificates. Indeed, any class of nonnegative polynomials containing the nonnegative constants can be used as the base class  (i.e., the class used to certify nonnegativity). This means that convergent hierarchies to address the solution of PO problems can be constructed with a wide range of classes of nonnegative polynomials beyond SOS polynomials, such as (but not limited to) DSOS, SDSOS, or SONC polynomials.
In turn, this results in freedom to choose different optimization techniques to address the solution of PO problems (e.g., linear and second-order cone optimization and geometric programming). This is particularly important given the lack of scalability of approaches based on the use of SOS polynomials for general PO problems.

We also derive a semi-sparse \germanF{} in which SOS polynomials are used as the base class. The notable feature of this result is that the inherent sparsity in the associated certificate allows to exploit sparsity in a PO problem that cannot be efficiently exploited by existing sparse certificates.

Another feature of the \germansF{} introduced in the article is that, unlike some of the most popular \germansF{} based on the use of SOS polynomials, the proposed \germansF{} are guaranteed to exist for cases in which the underlying set of interest is unbounded or compact but not Archimedean.

The positive numerical results presented in Sections~\ref{sec:SchNumerics}, \ref{sec:sparseNumerics}, and~\ref{sec:copocertexp} indicate the potential computational performance of the proposed \germansF{}  to solve PO problems. It is reasonable to expect that further developments will enhance the computational performance that can be achieved with the proposed methodology. In particular, advances in the use of both SOS and non-SOS classes of polynomials, for polynomial optimization, can help to
improve the numerical performance of the new \germansF{} we derive. In this direction, it is worth noting the improvements on the numerical use of DSOS and SDSOS polynomials in~\citep{ahmadi2017sum}, and in the use of SONC polynomials or the related
{\em sums of arithmetic-geometric exponential} (SAGE) functions in~\citep{dressler2022algebraic, magron2023sonc}. Also relevant is the novel class of SOS+SONC polynomials proposed in~\citep{dressler2023geometrical} that can be readily used as the base class in the proposed non-SOS \germansF{}.
These articles, together with
articles such as~\citep[see, e.g.,][to name a few]{dressler2023geometrical, wang2020second, hyperb1, roebers2021sparse, dressler2018optimization} show the continued interest in developing novel non-SOS solution approaches for PO problems.

\looseness -1
\section{Acknowledgments} We thank three anonymous referees for their constructive and thou\-ght\-ful comments which greatly helped to improve the article.
We express our formal gratitude to Bissan Ghaddar for sharing code that significantly expedited the implementation of the article's numerical experiments.

%\bibliographystyle{apalike}
%\bibliography{CopCert_BibTeX}

\end{document}